\documentclass{elsarticle}               

\usepackage{graphicx}          
\usepackage{amssymb}
\usepackage{amsmath}
\usepackage{subcaption}
\usepackage{xcolor}
\usepackage{breqn}
\usepackage{booktabs} 

\usepackage[hidelinks]{hyperref}
\usepackage{amsthm}
\newtheorem{assumption}{Assumption}
\newtheorem{proposition}{Proposition}
\newtheorem{remark}{Remark}

\newcommand{\vf}{\mathbf{f}}
\newcommand{\vx}{\mathbf{x}}
\newcommand{\vu}{\mathbf{u}}
\newcommand{\vh}{\mathbf{h}}

\newcommand{\vZ}{\mathbf{Z}}
\newcommand{\vz}{\mathbf{z}}
\newcommand{\vS}{\mathbf{S}}
\newcommand{\bv}{\mathbf{v}}
\newcommand{\vl}{\mathbf{l}}
\newcommand{\vQ}{\mathbf{Q}}
\newcommand{\vR}{\mathbf{R}}
\newcommand{\vk}{\mathbf{k}}
\newcommand{\vK}{\mathbf{K}}
\newcommand{\blambda}{\boldsymbol{\lambda}}
\newcommand{\bLambda}{\boldsymbol{\Lambda}}

\newcommand{\bo}{\boldsymbol{0}}
\newcommand{\cJ}{\mathcal{J}}
\newcommand{\tran}{^{\top}}
\makeatletter
\long\def\@maketablecaption#1#2{\@tablecaptionsize
  \setbox\@tempboxa\hbox{#1. #2}
  \ifdim \wd\@tempboxa >\hsize              
    \unhbox\@tempboxa\par                   
  \else                                     
    \global \@minipagefalse
    \hbox to\hsize{\hfil\box\@tempboxa\hfil}
  \fi}
\makeatother

\begin{document}

\begin{frontmatter}

\title{A Sequential Quadratic Programming Perspective on Optimal Control}

\tnotetext[footnoteinfo]{This paper was not presented at any IFAC 
meeting. Corresponding author: S.~Chakravorty.}

\author[TAMU]{Abhijeet}\ead{abhinir@tamu.edu}    
\author[TAMU]{Suman Chakravorty}\ead{schakrav@tamu.edu}              

\address[TAMU]{Texas A\& M University, College Station, TX, US.}  


\begin{abstract}                          
This paper investigates the performance of Newton's method, iterative Linear Quadratic Regulator (iLQR), and Differential Dynamic Programming (DDP) in solving discrete-time optimal control problems. We offer a unified perspective on these approaches, centered on the understanding that each method ultimately solves a sequence of quadratic programs. Building upon previous comparative works, this paper contributes additional mathematical explanations and results to the analysis.

In particular, it is shown that iLQR is a principled Sequential Quadratic Programming (SQP) approach, rather than merely an approximation of DDP that neglects Hessian terms. This characteristic guarantees that iLQR will always produce a cost-descent direction and converge to an optimum, under some mild assumptions. In contrast, Newton's method and DDP lack these guarantees, especially when initialized far from an optimum. A series of numerical examples are presented to corroborate the mathematical reasoning and analysis developed in the paper.
\end{abstract}

\end{frontmatter}

\section{Introduction}

The study of iterative trajectory optimization has a long history in nonlinear
control. Among the most influential methods are Differential Dynamic Programming
(DDP), introduced in the 1960s \cite{mayne1965ddp, JacobsonMayne1970DDP}, and the
iterative Linear Quadratic Regulator (iLQR), which later emerged as a simplified
yet powerful alternative \cite{todorov2005ilqg, tassa2012synthesis}. DDP is acclaimed as
Newton-type algorithm for trajectory optimization that exploits second-order
expansions of both the cost and the dynamics, achieving locally quadratic
convergence near an optimum \cite{mayne1965ddp, JacobsonMayne1970DDP,
jacobson1970ddp}. iLQR, by contrast, employs consistent first-order
approximations of the dynamics, making it more numerically reliable and
computationally efficient in practice \cite{todorov2005ilqg, tassa2012synthesis,D2C}.
Over the decades, these two approaches have become foundational tools for
nonlinear control, and trajectory optimization
\cite{BrysonHo69,Wang2025SearchFeedbackRL,inf_horizon1,inf_horizon2}.

A large body of literature has analyzed the theoretical properties and limitations
of both methods. Several works have emphasized the quadratic convergence of DDP
under ideal assumptions \cite{mayne1965ddp, JacobsonMayne1970DDP, Pantoja_SN, shoemaker1983ddp},
but have also noted the difficulties that arise in practice when the algorithm's
backward pass combines second-order expansions with first-order terms, which can
compromise stability and necessitate heavy regularization \cite{Pantoja_SN,LiaoShoemaker1992,feng2020optimization,NgangaWensing2021RAL,pardo2017time}.
An explicit connection between DDP and Newton's method was established in
\cite{Pantoja_SN,Murray1984DDP_Newton}, and it is argued in \cite{LiaoShoemaker1992} that, for
discrete-time optimal control, DDP can outperform Newton's method---retaining
strong local convergence while offering greater robustness and efficiency under
practical problem structures. For rigid-body systems, accelerated second-order DDP
has been explored \cite{NgangaWensing2021RAL}. Extensions such as the iterative
Linear Quadratic Gaussian (iLQG) \cite{todorov2005ilqg} and control-limited iLQR
\cite{tassa2014control} have reinforced the practical importance of iLQR in
stochastic and constrained settings.

The success of trajectory optimization methods has been demonstrated broadly
in robotics. Model Predictive Control (MPC) frameworks based on iLQR and DDP
have enabled real-time control of complex robotic systems
\cite{Neunert2018WholeBodyMPC, Farshidian2017EfficientMPC}. The development of
solvers such as ALTRO \cite{ALTRO} and ProxDDP \cite{ProxDDP} has
pushed the practical limits of constrained trajectory optimization for high
degree-of-freedom systems. 


In contrast to these developments, the adoption of DDP has been more
domain-specific. In aerospace engineering, DDP has been applied to trajectory
optimization in astrodynamics and spacecraft guidance, where its second-order
structure can be leveraged for high-precision planning \cite{Russell1, Russell2,
Aziz2019HDDP_CR3BP}. Beyond astrodynamics, DDP has been used for complex
robotic systems \cite{Oshin2022PDDP_RSS} and planetary landing problems
\cite{Noyes2021RobustEntryGuidance}, demonstrating its utility in solving
nonlinear, constrained, and multi-phase optimization problems. While effective
in these contexts, these studies also highlight the algorithm's sensitivity to
problem structure and its dependence on strong regularization, particularly
when the local quadratic models become inconsistent \cite{Pantoja_SN,
LiaoShoemaker1992}.

Although iLQR is often described as an approximation of DDP obtained by
neglecting second-order dynamic terms \cite{todorov2005ilqg}, this characterization
obscures its principled nature. In this paper, we revisit the solution of optimal
control problems through the lens of  Sequential Quadratic Programming
(SQP) \cite{boggs1995sqp, Nocedal2006NumericalOptimization}.The critical insight is that 
the Newton step in a constrained optimization problem is equivalent to solving a
Quadratic Program (QP), which reduces to an LQR problem in the context of optimal
control. Using this framework, we show that iLQR is not merely an approximation
of DDP, but rather a principled method that guarantees cost-descent directions
under mild assumptions---a property that neither DDP nor Newton's method can ensure
in general, particularly when far from an optimum \cite{Pantoja_SN,
LiaoShoemaker1992}. We further derive the exact Newton LQR formulation, and
show that DDP can be interpreted as an approximation of this exact Newton step
only near an optimum. Our empirical evaluations on pendulum and cart-pole swing-up
tasks, as well as several benchmark nonlinear problems, corroborate the SQP-based
analysis and demonstrate that iLQR delivers more reliable and efficient solutions
than DDP, often requiring far less computation due to the absence of dynamics
Hessian terms \cite{Wang2025SearchFeedbackRL,D2C}.

The remainder of this paper is organized as follows.
Section~\ref{sec:background} introduces the problem formulation and background on
the algorithms discussed in the paper. Section~\ref{sec:algorithms} presents a detailed comparison of these algorithms.
Section~\ref{sec:neighbouring_optima} presents neighbouring optimal solution as DDP equations at the optimal.
Section~\ref{sec:conclusion} concludes with insights drawn from these comparisons.

\section{Background}\label{sec:background}
We address the general class of discrete-time, finite-horizon optimal control problems. Let the state of a system at time step $t$ be $\vx_t \in \mathbb{R}^n$, and the control input be $\vu_t \in \mathbb{R}^m$. The system evolves over a finite time horizon of $T$ steps according to the nonlinear state-transition function $f$:
\begin{equation}
    \vx_{t+1} = \vf(\vx_t, \vu_t), \qquad t=0, 1, \cdots, T-1,
    \label{eq:dynamics}
\end{equation}
starting from a given initial state $\vx_0$.

The objective is to find a sequence of control inputs $\mathbf{U} = \{\vu_0, \dots, \vu_{T-1}\}$ that minimizes a total cost functional, $\cJ$. The optimal control problem can be formally stated as:
\begin{equation}
    \min_{\mathbf{U}} \cJ(U) = \sum_{t=0}^{T-1} c(\vx_t, \vu_t) + C_T(\vx_T),
    \label{eq:ocp}
\end{equation}
subject to the dynamics in \eqref{eq:dynamics}. To ensure the problem is well-posed for the optimization methods used in this work, we make the following assumption.

\begin{assumption}
The system dynamics function $\vf$, stage cost $c$, and terminal cost $C_T$ are twice continuously differentiable ($\mathcal{C}^2$) functions with respect to their arguments.
\label{assum:smoothness}
\end{assumption}
The prominent trajectory optimization methods that we analyze in this paper, including Differential Dynamic Programming (DDP), the iterative Linear Quadratic Regulator (iLQR), and Stagewise Newton methods, all rely on a common underlying principle. They iteratively refine a nominal trajectory by solving a local, second-order approximation of the original optimal control problem. Constructing this approximation requires the derivatives of the dynamics and cost functions evaluated along the nominal trajectory.

Let $(\bar{\vx}_t, \bar{\vu}_t)$ for $t=0, \dots, T-1$ denote such a nominal trajectory, which is a dynamically feasible sequence satisfying \eqref{eq:dynamics}. We define the state and control perturbations, or variations, from this nominal trajectory as:
\begin{equation}
    \delta \vx_t := \vx_t - \bar{\vx}_t \quad \text{and} \quad \delta \vu_t := \vu_t - \bar{\vu}_t.
\end{equation}

For notational brevity, we adopt a compact subscript notation for the partial derivatives. The first-order derivatives (Jacobians) and second-order derivatives (Hessians) of the dynamics and cost functions are written as:
\[
\vf_\vx, \quad \vf_\vu, \quad \vf_{\vx\vx}, \quad \vf_{\vu\vu}, \quad \vf_{\vu\vx}, \quad
c_\vx, \quad c_\vu, \quad c_{\vx\vx}, \quad c_{\vu\vu}, \quad c_{\vx\vu}.
\]
Derivatives of the terminal cost are denoted similarly, e.g., $C_{T,\vx}$ and $C_{T,\vx\vx}$. Unless stated otherwise, all such derivatives are evaluated at the corresponding time step along the nominal trajectory, i.e., $\vf_\vx \equiv \frac{\partial \vf}{\partial \vx}(\bar{\vx}_t, \bar{\vu}_t)$. Finally, we denote the transpose of a vector or matrix $\mathbf{A}$ as $\mathbf{A}\tran$.

A unified perspective for understanding and comparing the trajectory optimization algorithms in this paper is offered by the framework of Sequential Quadratic Programming (SQP). We therefore begin by developing the fundamental concepts of the SQP approach as applied to the general optimal control problem described in eq.~\eqref{eq:ocp}. 

\subsection{Sequential Quadratic Programming}
In this section, we outline the basic SQP program for future references. Consider the generic optimization problem:
\begin{align*}
    &\min_{\vz} g(\mathbf{z}),\\
    \text{subject to:}~~&\mathbf{h}(\mathbf{z}) = \boldsymbol{0}.
\end{align*}
This constrained optimization can be converted to an unconstrained optimization using an augmented Lagrangian as:
\begin{align*}
    g(\mathbf{z}) + \bLambda\tran \mathbf{h}(\mathbf{z}) = 0.
\end{align*}

This gives the necessary conditions as:
\[
\frac{\partial \mathcal{L}}{\partial \mathbf{z}} = \boldsymbol{0} 
\quad \text{and} \quad 
\frac{\partial \mathcal{L}}{\partial \boldsymbol{\bLambda}} = \boldsymbol{0},
\]
yielding:
\[
g_{\mathbf{z}} +  \vh_{\mathbf{z}}\tran \bLambda^*= \boldsymbol{0},\;
\vh(\mathbf{z}^*) = \boldsymbol{0},
\]
at a minimum (stationary point) \((\vz^*, \bLambda^*)\).

To solve the problem, the second-order Newton method suggests to start with a nominal trajectory: \((\bar{\vz}, \bar{\bLambda})\) and then expand the Lagrangian 
\(\mathcal{L}\) to second order around the nominal.
\begin{align*}
\mathcal{L}(\vz, \bLambda) = g + \bLambda \tran \mathbf{h} 
= \big( \bar{g} + \bar{\bLambda}\tran\bar{\mathbf{h}} \big) 
+ (g_\vz \delta \vz + \bar{\bLambda}\tran \mathbf{h}_\vz \delta \vz + \delta \bLambda\tran\bar{\mathbf{h}}) \\
+ \frac{1}{2} \delta \vz\tran \big( g_{\vz  \vz} + \bar{\bLambda} \otimes \mathbf{h}_{\vz \vz} \big) \delta \vz 
+  \delta \vz\tran \, \mathbf{h}_\vz \delta \bLambda 
 \\
\equiv \bar{\mathcal{L}} + \Delta \mathcal{L},
\end{align*}
where $\bar{g} = g(\bar{\vz})$, $\bar{\mathbf{h}} = \mathbf{h}(\bar{\vz})$, $\bar{\mathcal{L}} = \bar{g}+\bar{\bLambda}\tran\bar{\mathbf{h}}$ and $\otimes$ denote a tensor product.


In a standard SQP approach, one follows iterative Newton's method and optimizes the \((\Delta \mathcal{L})\) part in terms of the perturbations \((\delta \vz, \delta \bLambda)\) and updates the solution iteratively as:
\[
\vz = \bar{\vz} + \delta \vz, 
\quad 
\bLambda = \bar{\bLambda} + \delta \bLambda.
\]

The necessary conditions for obtaining the optimum of \(\Delta \mathcal{L}\) in terms of \(\delta x, \delta \lambda\) are obtained by setting
\[
\frac{\partial (\Delta \mathcal{L})}{\partial (\delta \vz)} = 0 
\quad \text{and} \quad 
\frac{\partial (\Delta \mathcal{L})}{\partial (\delta \bLambda)} = 0,
\]
yielding:
\[
\begin{bmatrix}
g_{\vz \vz} + \bar{\bLambda}\otimes \mathbf{h}_{\vz \vz} && \mathbf{h}_{\vz}\tran \\
\mathbf{h}_{\vz} && \boldsymbol{0}
\end{bmatrix}
\begin{bmatrix}
\delta \vz \\[6pt]
\delta \bLambda
\end{bmatrix}
=
\begin{bmatrix}
- g_{\vz} - \mathbf{h}_{\vz}\tran\bar{\bLambda} \\[6pt]
-\bar{\mathbf{h}}
\end{bmatrix}
\]

However, taking the \( \mathbf{h}_\vz\tran \bar{\bLambda}\) term to the LHS gives:
\[
\begin{bmatrix}
g_{\vz \vz} + \bar{\bLambda}\otimes \mathbf{h}_{\vz \vz} && \mathbf{h}_{\vz}\tran \\
\mathbf{h}_{\vz} && \boldsymbol{0}
\end{bmatrix}
\begin{bmatrix}
\delta \vz \\[6pt]
\bar{\bLambda} + \delta \bLambda
\end{bmatrix}
=
\begin{bmatrix}
- g_{\vz} \\[6pt]
-\bar{\mathbf{h}}
\end{bmatrix}
\]
Here, note that \(\bLambda = \bar{\bLambda} + \delta \bLambda\) is simply the new \(\bLambda\) (the updated $\bLambda$ at new iteration).

However, the necessary conditions above are simply the solution to the QP problem:
\begin{align}
\min_{\delta \vz} \; g_\vz \delta \vz + \tfrac{1}{2} \delta \vz\tran \big(g_{\vz\vz} + \bar{\bLambda} \otimes \mathbf{h}_{\vz \vz}\big) \delta \vz \nonumber \\
\text{s.t.} \quad \mathbf{h}_{\vz} \delta \vz + \bar{\mathbf{h}} = 0. \label{N-QP}
\end{align}

Hence, the Newton step can be accomplished by solving the above QP for the solution \((\delta \vz^q, \bLambda^q)\) and updating the solution as:
\[
\vz = \bar{\vz} + \delta \vz^q, 
\quad 
\bLambda = \bLambda^q.
\]

Thus, the original constrained optimization can be solved by solving a sequence of QP problems above, and this will be termed as the Newton SQP approach. \\
However, in general, it is not advisable to use the ``true" Hessian of the Lagrangian \(\big(g_{\vz \vz} + \bar{\bLambda}\otimes \mathbf{h}_{\vz \vz}\big)\) in the QP unless one is close to an optimum.

To see this, note that in the constrained case, the descent direction 
$\delta \vZ = \begin{pmatrix} \delta \vz \\ \delta \bLambda \end{pmatrix}$ 
is given by:
\[
\delta \vZ = (\nabla^2 \mathcal{L})^{-1} \nabla \mathcal{L},
\]
and
\[
\nabla^2 \mathcal{L} = 
\begin{bmatrix}
g_{\vz \vz} + \bar{\bLambda} \otimes \mathbf{h}_{\vz \vz} && \mathbf{h}_{\vz} \\
\mathbf{h}_{\vz} && \boldsymbol{0}
\end{bmatrix}
\quad \nsucc 0
\]
if the current solution is far from an optimum.  

If the Hessian $\nabla^2 \mathcal{L} \succ 0$, the Newton direction can be assured to be a descent direction, since then:
\[
\delta \vZ' \nabla \mathcal{L} = \nabla \mathcal{L}' (\nabla^2 \mathcal{L})^{-1} \nabla \mathcal{L} > 0.
\]
However, if $\nabla^2 \mathcal{L} \nsucc 0$, there is no such guarantee.  

Thus, in order to be able to use the SQP approach, it becomes critical that one be able to ensure that the QP solution is a descent direction for the cost function.  

Consider now a modified problem (QP):
\begin{align}
\min_{\delta x} \; g_\vz \delta \vz + \tfrac{1}{2} \, \delta \vz\tran g_{\vz \vz} \delta \vz \nonumber\\
\mathbf{h}_{\vz} \delta \vz + \bar{\mathbf{h}} = 0. \label{M-QP}
\end{align}

We make the assumption that $\mathbf{h}(\bar{\vz}) = \bar{\mathbf{h}} = \boldsymbol{0}$, i.e., our current estimate is feasible.  

The necessary conditions for this QP are:
\begin{align}
g_\vz + g_{\vz \vz} \delta \vz + \mathbf{h}_{\vz}\tran \bLambda &= \boldsymbol{0},  \\
\bar{\mathbf{h}} + \mathbf{h}_\vz \delta \vz &= \boldsymbol{0}. 
\end{align}



Since $g(\vz)$ is a cost function by design, it is a reasonable assumption that $g_{\vz \vz} \succ 0$ for $\vz$.  

Hence, the solution to the modified QP \eqref{M-QP} is always a descent direction for the cost function.  
Furthermore,
\[
\mathbf{h}(\bar{\vz} + \delta \vz) \approx \mathbf{h}(\bar{\vz}) + \mathbf{h}_\vz \delta \vz = \boldsymbol{0}
\]
if $\delta \vz$ is small, and thus the update $\vz = \bar{\vz} + \delta \vz$ also satisfies the constraint. If $\delta \vz$ is not small enough, we can do a line search using a parameter $\alpha \in (0,1]$ such that
\[
\mathbf{h}(\vz) = \mathbf{h}(\bar{\vz} + \alpha \delta \vz) \approx \mathbf{h}(\bar{\vz}) + \alpha \mathbf{h}_\vz \delta \vz = 0,
\]
and thus $\vz = \bar{\vz} + \alpha \delta \vz$ is feasible.  

Hence, if we have a feasible solution, the modified QP always gives a descent direction for the cost while ensuring that the solution remains feasible if allied with a suitable line search approach.  
The development above can be summarized as the following result:  

\medskip

\begin{proposition} 
Consider the modified QP \eqref{M-QP}. Given $g_{\vz \vz} \succ 0$, the solution to the QP, $\delta \vz$, is always a descent direction for the cost function $g(\vz)$. Furthermore, the new solution $\vz = \bar{\vz} + \alpha \delta \vz$ can be made feasible, i.e., $\mathbf{h}(\vz) = \boldsymbol{0}$, by a suitable choice of the line-search parameter, $\alpha \in (0,1]$.
\end{proposition}

\subsection{Background of iterative perturbation feedback methods}

The motivation for introducing SQP in the preceding section was to establish a link to the algorithms compared in this work. Before presenting that comparison, we first develop the necessary background for these algorithms. In the following, we reiterate how Differential Dynamic Programming (DDP), stagewise Newton methods, and the iterative Linear Quadratic Regulator (iLQR) utilize the sparsity structure inherent in the optimal control problem described by eq.~\eqref{eq:ocp} and eq.~\eqref{eq:dynamics}. This exploitation of sparsity leads to significantly reduced computational complexity compared to conventional Sequential Quadratic Programming (SQP).

To begin with, consider a feasible control sequence and a corresponding trajectory $(\bar{\vx}_t, \bar{\vu}_t)$ can be constructed by evolving the dynamics.
Next, the augmented cost (using eqs. \eqref{eq:ocp} and \eqref{eq:dynamics}) can be written as:
\begin{equation}
    \min_{\vx_t, \vu_{t}, \blambda_t} \sum_{t=0}^{T-1}\Big( c(\vx_{t}, \vu_{t}) + \blambda_{t+1}\tran(\vf(\vx_{t}, \vu_{t}) - \vx_{t+1}) \Big) + C_{T}(\vx_{T})
\end{equation}
where $\blambda_{t+1} \in \mathbb{R}^{n}$ are the Lagrange multipliers/costates. 

In addition, let $(\bar{\vx}_t, \bar{\vu}_t, \bar{\blambda}_t)$ be the current guess to the problem. This augmented cost can be expanded about $(\bar{\vx}_{t}, \bar{\vu}_{t}, \bar{\blambda}_{t})$ to get
\begin{dmath}
\bar{\cJ}+\delta\cJ+\delta^{2}\cJ
=\min_{\delta \vx_t,\;\delta \vu_t,\;\delta \blambda_t}\Bigg(
\sum_{t=0}^{T-1}\Big(
c(\bar{\vx}_{t},\bar{\vu}_{t})
+\bar{\blambda}_{t+1}\tran\big(\vf(\bar{\vx}_{t},\bar{\vu}_{t})-\bar{\vx}_{t+1}\big)
\Big)
+ C_{T}(\bar{\vx}_{T}) 
+\sum_{t=0}^{T-1}\Big(
c_{\vx_{t}}\delta \vx_{t}
+c_{\vu_{t}}\delta \vu_{t}
+\delta \blambda_{t+1}\tran\big(\bar{\vx}_{t+1}-\vf(\bar{\vx}_{t},\bar{\vu}_{t})\big)
+\bar{\blambda}_{t+1}\tran\big(
\vf_{\vx_{t}}\delta \vx_{t}
+\vf_{\vu_{t}}\delta \vu_{t}
-\delta \vx_{t+1}
\big)
+\tfrac{1}{2}\delta \vx_{t}\tran c_{\vx_{t} \vx_{t}} \delta \vx_{t}
+\tfrac{1}{2}\delta \vu_{t}\tran c_{\vu_{t} \vu_{t}} \delta \vu_{t}
+\delta \vx_{t}\tran c_{\vx_{t}\vu_{t}} \delta \vu_{t}
+\bar{\blambda}_{t+1}\tran\Big(
\tfrac{1}{2}\delta \vx_{t}\tran f_{\vx_{t} \vx_{t}} \delta \vx_{t}
+\delta \vx_{t}\tran \vf_{\vx_{t} \vu_{t}} \delta \vu_{t}
+\tfrac{1}{2}\delta \vu_{t}\tran f_{\vu_{t} \vu_{t}} \delta \vu_{t}
\Big)
+\delta \blambda_{t+1}\tran\big(
\vf_{\vx_{t}}\delta \vx_{t}
+\vf_{\vu_{t}}\delta \vu_{t}
-\delta \vx_{t+1}
\big)
\Big)
+ C_{T,\vx}\delta \vx_{T}
+\tfrac{1}{2}\delta \vx_{T}' C_{T,\vx\vx} \delta \vx_{T}
\Bigg).
\label{eq:per_optim}
\end{dmath}
The necessary conditions for the above QP are:
\begin{subequations}
    \begin{dmath}
        \vf(\bar{\vx}_t,\bar{\vu}_t) + \vf_{\vx_t}\delta \vx_t + \vf_{\vu_t}\delta \vu_t - \delta \vx_{t + 1} - \bar{\vx}_{t+1} = \bo
        \label{eq:lin_dyn}
    \end{dmath}
    \begin{dmath}
        c_{\vx_t} + c_{\vx_t \vx_t}\delta \vx_t + c_{\vx_t \vu_t}\delta \vu_t
        +(\vf_{\vx_t \vx_t} \delta \vx_{t} + \vf_{\vx_t \vu_{t}}\delta \vu_{t})\tran\bar{\blambda}_{t+1} + \vf_{\vx_t}\tran(\bar{\blambda}_{t+1} + \delta \blambda_{t+1})
        -(\bar{\blambda}_t + \delta \blambda_t) = \bo
        \label{eq:state_DDP}
    \end{dmath}
    \begin{dmath}
        c_{\vu_t} + c_{\vu_t \vu_t}\delta \vu_t + c_{\vu_t \vx_t} \delta \vx_t
        +(\vf_{\vu_t \vu_t} \delta \vu_{t} + \vf_{\vu_t, \vx_{t}}\delta \vx_{t})\tran \bar{\blambda}_{t+1} + \vf_{\vu_t}\tran(\bar{\blambda}_{t+1} + \delta \blambda_{t+1}) = \bo
        \label{eq:con_DDP}
    \end{dmath}    
\end{subequations}
with the boundary/transversality condition:
\begin{align}
    \bar{\blambda}_{T} + \delta \blambda_{T} = C_{T,\vx} + C_{T, \vx\vx}\delta \vx_{T}, \nonumber \\
    \bar{\blambda}_{T} = C_{T,\vx} \quad \text{and} \quad \delta \blambda_{T} = C_{T,\vx\vx} \delta \vx_{T}.
    \label{eq:transver}
\end{align}
Conventional SQP solves this problem by posing this problem as:
\begin{equation*}
    \mathbf{H}\begin{bmatrix}
        \delta \vx_t\\
        \delta \vu_t\\
        \delta \blambda_t
    \end{bmatrix} + \mathbf{g} = \boldsymbol{0} \Rightarrow \begin{bmatrix}
        \delta \vx_t\\
        \delta \vu_t\\
        \delta \blambda_t
    \end{bmatrix} = -\mathbf{H}^{-1}\mathbf{g},
\end{equation*}
where $\mathbf{H}$ is the hessian and $\mathbf{g}$ is the gradient respectively. It gives the cushion of adding additional constraints, other than the dynamics in eq.~\eqref{eq:dynamics}.
However, the matrix $\mathbf{H}$ is of the order $(m+2n)T$, and hence, inverting this matrix is computationally very heavy.\\ 
A class of algorithms (DDP/Stagewise Newton/iLQR), referred to as ``Iterative perturbation feedback methods" in this paper, use the sparsity of this optimization problem. These algorithm use a forward-backward structure to exploit the sparsity and solve this problem with a time complexity $\mathcal{O}(T)$. There is a forward and a backward pass composing these algorithms. In the forward pass, the dynamics eq. \eqref{eq:dynamics} is propagated while computing the update to control. While in the backward pass, one implicitly solves for the multipliers and a linear feedback update to the control. The solution is made feasible using a suitable line search parameter $\alpha$.

We give an overview of these algorithms starting with ``DDP". First, it starts with an initial guess of $\{\vu_{t}\}_{t=0}^{T-1}$ and initializes the trajecotry using the dynamics $\bar{\vx}_{t+1} = \vf(\bar{\vx}_t,\bar{\vu}_t)$ in the forward pass. Following this, the control update $\delta \vu_t$ is solved by leveraging the sparsity of variables in the backward pass. The backward pass begins at the last time step  with the transversality condition in eq. \eqref{eq:transver} and then eq. \eqref{eq:con_DDP} is used to find the control update. Furthermore, $\delta \vx_{T}$ is replaced by linearized dynamics using eq. \eqref{eq:lin_dyn} in terms of $\delta \vu_{t-1}$  and $\delta \vx_{t-1}$. The equation which follows from \eqref{eq:con_DDP} at time-step $T-1$ is
\begin{dmath*}
    c_{\vu_{T-1}} + c_{\vu_{T-1}\vu_{T-1}}\delta \vu_{T-1} + c_{\vu_{T-1}\vx_{T-1}}\delta \vx_{T-1} + (\vf_{\vu_{T-1}\vu_{T-1}}\delta \vu_{T-1} + \vf_{\vu_{T-1}\vx_{T-1}}\delta \vx_{T-1})\tran\bar{\blambda}_{T} +
    \vf_{\vu_{T-1}}\tran(\bar{\blambda}_{T} + \delta \blambda_{T}) = \bo.
\end{dmath*}
Using the transversality condition in eq. \eqref{eq:transver}, 
\begin{dmath*}
    c_{\vu_{T-1}} + c_{\vu_{T-1}\vu_{T-1}}\delta \vu_{T-1} + c_{\vu_{T-1}\vx_{T-1}}\delta \vx_{T-1}+ (\vf_{\vu_{T-1}\vu_{T-1}}\delta \vu_{T-1} + \vf_{\vu_{T-1}\vx_{T-1}}\delta \vx_{T-1})\tran C_{T,\vx}
    +\vf_{\vu_{T-1}}\tran(C_{T,\vx} + C_{T,\vx\vx}\delta \vx_{T}) = \bo.
\end{dmath*}
Furthermore, using the linearized dynamics in eq. \eqref{eq:lin_dyn} and replacing $\delta \vx_{T}$ in the above equation, we get
\begin{dmath*}
    c_{\vu_{T-1}} + c_{\vu_{T-1}\vu_{T-1}}\delta \vu_{T-1} + c_{\vu_{T-1}\vx_{T-1}}\delta \vx_{T-1} 
    + (\vf_{\vu_{T-1}\vu_{T-1}}\delta \vu_{T-1} + \vf_{\vu_{T-1}\vx_{T-1}}\delta \vx_{T-1})\tran C_{T,\vx}
    +\vf_{\vu_{T-1}}'(C_{T,\vx} + C_{T,\vx\vx}(\vf_{\vx_{T-1}} \delta \vx_{T-1} + \vf_{\vu_{T-1}}\delta \vu_{T-1})) = \bo.
\end{dmath*}
The above equation is linear in terms of $\delta \vu_{T-1}$ and the solution to it is:
\begin{dmath}
    \delta \vu_{T-1} = 
    -\Big{(} c_{\vu_{T-1}\vu_{T-1}} + \vf_{\vu_{T-1}\vu_{T-1}} \otimes C_{T,\vx} + \vf_{\vu_{T-1}}\tran C_{T,\vx\vx}\vf_{\vu_{T-1}}\Big{)}^{-1}
    \Bigg{(}c_{\vu_{T-1}} + \vf_{\vu_{T-1}}\tran C_{T,x} +
    \Big{(} c_{\vu_{T-1}\vx_{T-1}} + \vf_{\vu_{T-1}\vx_{T-1}} \otimes C_{T,\vx} + \vf_{\vu_{T-1}}C_{T,\vx\vx}\vf_{\vx_{T-1}}\Big{)}\delta \vx_{T-1}\Bigg{)}.
\end{dmath}
The control update in the above equation can be written in short as
\begin{align}
    \delta \vu_{T-1} = -(Q_{\vu\vu}^{T-1})^{-1}(Q_{\vu}^{T-1} + Q_{\vu\vx}^{T-1}\delta \vx_{T-1}).
    \label{eq:con_terminal}
\end{align}
The above term is the linear feedback for the DDP. The terms $Q_{\vu\vu}$,$Q_{\vu\vx}$ and $Q_{\vu}$ are
\begin{align*}
    &Q_{\vu\vu}^{T-1} = c_{\vu_{T-1}\vu_{T-1}} + \vf_{\vu_{T-1}\vu_{T-1}} \otimes C_{T,\vx} + \vf_{\vu_{T-1}}\tran C_{T,\vx\vx}\vf_{\vu_{T-1}}.\\
    &Q_{\vu\vx}^{T-1} = c_{\vu_{T-1}\vx_{T-1}} + \vf_{\vu_{T-1}\vx_{T-1}} \otimes C_{T,\vx} + \vf_{\vu_{T-1}}C_{T,\vx\vx}\vf_{\vx_{T-1}}.\\
   &Q_{\vu}^{T-1} = c_{\vu_{T-1}} + \vf_{\vu_{T-1}}\tran C_{T,\vx}.
\end{align*}

Hence, eq. \eqref{eq:con_DDP} is solved to get a linearized feedback control update at the terminal step, $\delta \vu_{T-1} = \vk_{T-1} + \vK_{T-1}\delta \vx_{T-1}$. This update is used in eq. \eqref{eq:state_DDP}, along with linearized dynamics from eq. \eqref{eq:lin_dyn}, to get a backward pass on $\bar{\blambda}_{T}$ (represented as $\bv_T$) and $\mathbf{S}_{T}$, where $\delta \lambda_{T} = \mathbf{S}_{T} \delta \vx_{T}$.
It comes from eq. \eqref{eq:state_DDP}:
\begin{dmath*}
    c_{\vx_{T-1}} + c_{\vx_{T-1}\vx_{T-1}}\delta \vx_{T-1} + c_{\vx_{T-1}\vu_{T-1}}\delta \vu_{T-1}
    +(\vf_{\vx_{T-1}\vx_{T-1}}\delta \vx_{T-1}+ \vf_{\vx_{T-1}\vu_{T-1}}\delta \vu_{T-1})\tran \bar{\blambda}_{T} + \vf_{\vx_{T-1}}\tran (\bar{\blambda}_{T} + \delta \blambda_{T})
    -(\bar{\blambda}_{T-1} + \delta \blambda_{T-1}) = \boldsymbol{0}.
\end{dmath*}
Substituting the value of control $\delta \vu_{T-1}$ from eq. \eqref{eq:con_terminal} and using the transversality condition one gets
\begin{dmath*}
    c_{\vx_{T-1}} + c_{\vx_{T-1}\vx_{T-1}}\delta \vx_{T-1} + c_{\vx_{T-1}\vu_{T-1}}(-Q_{\vu\vu}^{-1}(Q_{\vu}^{T-1} + Q_{\vu\vx}^{T-1}\delta \vx_{T-1})) 
    +(\vf_{\vx_{T-1}\vx_{T-1}}\delta \vx_{T-1}+ \vf_{\vx_{T-1}\vu_{T-1}}(-Q_{\vu\vu}^{-1}(Q_{\vu}^{T-1} + Q_{\vu\vx}^{T-1}\delta \vx_{T-1})))\tran C_{T,\vx} 
    +\vf_{\vx_{T-1}}\tran (C_{T,\vx} + C_{T,\vx\vx}\delta \vx_{T})
    -(\bar{\blambda}_{T-1} + \delta \blambda_{T-1}) = \bo.
\end{dmath*}
Using linearized dynamics for $\delta \vx_{T}$ from eq. \eqref{eq:lin_dyn} and replacing the control from eq. \eqref{eq:con_terminal}, the above expression becomes
\begin{dmath*}
    c_{\vx_{T-1}} + c_{\vx_{T-1}\vx_{T-1}}\delta \vx_{T-1} + c_{\vx_{T-1}\vu_{T-1}}(-(Q_{\vu\vu}^{T-1})^{-1}(Q_{\vu}^{T-1} + Q_{\vu\vx}^{T-1}\delta \vx_{T-1})) 
    +(\vf_{\vx_{T-1},x_{T-1}}\delta \vx_{T-1}+ \vf_{\vx_{T-1}\vu_{T-1}}(-(Q_{\vu\vu}^{T-1})^{-1}(Q_{\vu}^{T-1} + Q_{\vu\vx}^{T-1}\delta \vx_{T-1})))\tran C_{T,\vx} 
    + \vf_{\vx_{T-1}}\tran (C_{T,\vx} + C_{T,\vx\vx}(\vf_{\vx_{T-1}}\delta \vx_{T-1} + \vf_{\vu_{T-1}}
    (-(Q_{\vu\vu}^{T-1})^{-1}(Q_{\vu}^{T-1} + Q_{\vu\vx}^{T-1}\delta \vx_{T-1}))))
    -(\bar{\blambda}_{T-1} + \delta \blambda_{T-1}) = \bo.
    \label{eq:back_pass1}
\end{dmath*}
It is evident from the above equation that $\bar{\blambda}_{T-1}$ and $\delta \blambda_{T-1}$ can be separated using the form $\bv_{T-1} + \vS_{T-1}\delta \vx_{T-1}$. Separating out the constant terms gives
\begin{dmath*}
    c_{\vx_{T-1}} - c_{\vx_{T-1}\vu_{T-1}}(Q_{\vu\vu}^{T-1})^{-1}Q_{\vu}^{T-1} - (\vf_{\vx_{T-1}\vu_{T-1}}(Q_{\vu\vu}^{T-1})^{-1}Q_{\vu}^{T-1})\tran C_{T,\vx}
    +\vf_{\vx_{T-1}}\tran C_{T,\vx} - \vf_{\vx_{T-1}}\tran C_{T,\vx\vx}\vf_{\vu_{T-1}}(Q_{\vu\vu}^{T-1})^{-1}Q_{\vu}^{T-1} - \bv_{T-1} = \bo.
\end{dmath*}
Rearranging the terms give,
\begin{dmath*}
    c_{\vx_{T-1}} + \vf_{\vx_{T-1}}\tran C_{T,\vx}
    -(c_{\vx_{T-1}\vu_{T-1}} + C_{T,\vx}\otimes \vf_{\vx_{T-1}\vu_{T-1}} + \vf_{\vx_{T-1}}\tran C_{T,\vx\vx}\vf_{\vu_{T-1}})(Q_{\vu\vu}^{T-1})^{-1}Q_{\vu}^{T-1} = \bv_{T-1},
\end{dmath*}
or in short,
\begin{dmath}
    \bv_{T-1} = Q_{\vx}^{T-1} - Q_{\vx\vu}^{T-1}(Q_{\vu\vu}^{T-1})^{-1}Q_{\vu}^{T-1},
\end{dmath}
where
\begin{align*}
    &Q_{\vx}^{T-1} = c_{\vx_{T-1}} + \vf_{\vx_{T-1}}\tran C_{T,\vx}.\\
    &Q_{\vx\vu}^{T-1} =  c_{\vx_{T-1}\vu_{T-1}} +  C_{T,\vx}\otimes \vf_{\vx_{T-1}\vu_{T-1}}+ \vf_{\vx_{T-1}}\tran C_{T,\vx\vx}\vf_{\vu_{T-1}}.
\end{align*}
Similarly, the terms involving $\delta \vx_{T-1}$ in eq. \eqref{eq:back_pass1} are:
\begin{dmath*}
    c_{\vx_{T-1}\vx_{T-1}} - c_{\vx_{T-1}\vu_{T-1}}(Q_{\vu\vu}^{T-1})^{-1}Q_{\vu\vx} + (\vf_{\vx_{T-1}\vx_{T-1}}\otimes C_{T,\vx})
    - (\vf_{\vx_{T-1}\vu_{T-1}}\otimes C_{T,\vx}  Q_{\vu\vu}^{-1}Q_{\vu\vx}) + \vf_{\vx_{T-1}}\tran C_{T,\vx\vx}\vf_{\vx_{T-1}} 
    - \vf_{\vx_{T-1}}\tran C_{T,\vx\vx}\vf_{\vu_{T-1}}Q_{\vu\vu}^{-1}Q_{\vu\vx} = \vS_{T-1},
\end{dmath*}
or in short,
\begin{align}
    \vS_{T-1} = Q_{\vx\vx}^{T-1} - Q_{\vx\vu}^{T-1}(Q_{\vu\vu}^{T-1})^{-1}Q_{\vu\vx}^{T-1}
\end{align}
where
\begin{dmath*}
    Q_{\vx\vx}^{T-1} = c_{\vx_{T-1}\vx_{T-1}} +  \vf_{\vx_{T-1} \vx_{T-1}} \otimes C_{T,\vx} +  \vf_{\vx_{T-1}}\tran C_{T,\vx\vx}\vf_{\vx_{T-1}}
\end{dmath*}
Now, this can be generalized for every step to get $\bv_{t}$ and $\vS_{t}$ and these can be used to calculate $\vu_{t} = \vk_{t} + \vK_{t}\delta \vx_{t}$. 
In  a nutshell, eq. \eqref{eq:state_DDP} becomes an equation which can be separated into constants and linear terms in $\delta x_{t}$. This is then used to generate the backward pass on $\bv_t (\bar{\blambda}_{t})$ and $\vS_{t}$. This process is repeated to obtain all $\delta \vu_{t}$ and update the trajectory. The general equations for backward pass are:
\begin{align}
    \bv_{t} = Q_{\vx}^{t} - (Q_{\vu}^{t})\tran (Q_{\vu\vu}^{t})^{-1}Q_{\vu\vx}^{t}, \label{eq:back_pass_DDP2}\\
    \vS_{t} = Q_{\vx\vx}^{t} - Q_{\vx\vu}^{t}(Q_{\vu\vu}^{t})^{-1}Q_{\vu\vx}^{t},
    \label{eq:back_pass_DDP}
\end{align}
where
\begin{align*}
    &Q_{\vx}^{t} = c_{\vx_{t}} +  \vf_{\vx_t}\tran \bv_{t+1} \\
    &Q_{\vu}^{t} = c_{\vu_{t}} +  \vf_{\vu_t} \tran \bv_{t+1}\\
    &Q_{\vx\vx}^{t} = c_{\vx_t\vx_t} + \vf_{\vx_t}\tran \vS_{t+1}\vf_{\vx_t} + \bv_{t+1}\otimes  \vf_{\vx_t \vx_t} \\
    &Q_{\vu\vu}^{t} = c_{\vu_t\vu_t} + \vf_{\vu_t}\tran \vS_{t+1}\vf_{\vu_t} + \bv_{t+1}\otimes \vf_{\vu_t \vu_t}\\
    &Q_{\vu\vx}^{t} = c_{\vu_t\vx_t} + \vf_{\vu_t}\tran \vS_{t+1}\vf_{\vx_t} + \bv_{t+1} \otimes \vf_{\vu_t \vx_t}
\end{align*}
with the terminal conditions:
\begin{align*}
    \vS_{T} = C_{T,\vx\vx} \quad \text{and} \quad \bv_{T} = C_{T,\vx}.
\end{align*}
The term $\vk_{t} = -(Q_{\vu\vu}^{t})^{-1}Q_{\vu}^{t}$ and $\vK_{t} = -(Q_{\vu\vu}^{t})^{-1}Q_{\vu\vx}^{t}$. DDP computes $\vk_{t}$ and $\vK_{t}$ for all time steps and then updates the control $\vu_{t} = \bar{\vu}_{t} + \alpha\vk_{t} + \vK_{t}\delta \vx_{t}$ where $\delta \vx_{t}$ is the perturbation given by the system after applying $\delta \vu_{t-1}$ which is $\delta \vx_{t} =  \vf(\bar{\vx}_{t-1} + \delta \vx_{t-1}, \bar{\vu}_{t-1} + \delta \vu_{t-1}) - \vf(\bar{\vx}_{t-1}, \bar{\vu}_{t-1})$, with $\delta \vx_{0} = 0$ as the initial state, $\vx_{0}$, is given.

The perturbed cost can be simplified to find the expected reduction in cost by substituting $\delta \vu_{t} = \alpha \vk_{t} + \vK_{t}\delta \vx_{t}$ and the backward pass equations given in eq. \eqref{eq:back_pass_DDP} to get
\begin{equation}
    \delta \cJ^{*} + \delta^{2}\cJ^{*} = -\Big(\alpha - \frac{\alpha^{2}}{2}\Big)\sum_{t=0}^{T-1}(Q_{\vu}^{t})\tran (Q_{\vu\vu}^{t})^{-1}Q_{\vu}^{t}
    \label{eq:exp_cost_red}
\end{equation}
which is the expected optimal reduction in cost. One can see that $Q_{\vu\vu}^{t}$ has to be positive definite for a guarantee in cost reduction at each time step $t$. 
\begin{remark}
    The matrix $Q_{\vu\vu}^{t}$ in DDP is often not positive definite due to the second order terms $\vf_{\vx_t\vx_t},\vf_{\vu_t\vx_t}$ and $\vf_{\vu_t\vu_t}$, which depends on both the system dynamics and the current nominal trajectory. This poses the same problems as the QP in eq.~\eqref{N-QP}. As a result, DDP typically requires substantial regularization. This issue also arises in the ``Stagewise Newton" approach, as will be discussed next.
\end{remark}

The method ``Stagewise Newton" solves the same set of necessary conditions as in eqs. \eqref{eq:lin_dyn}-\eqref{eq:con_DDP}, like DDP. However, there are two key differences. In the forward pass, it uses linearized dynamics to update the control. So, $\vu_{t} = \bar{\vu}_{t} + \alpha \vk_{t} + \vK_{t}\delta\hat{\vx}_{t}$, where $\delta \hat{\vx}_{t} = \vf_{\vx_{t-1}}\delta \hat{\vx}_{t-1} + \vf_{\vu_{t-1}}\delta \vu_{t-1}$. 

In addition to this, one replaces $\bar{\blambda}_{t} + \delta \blambda_{t} $ to a new variable $\vl_{t}$ in eqs. \eqref{eq:state_DDP} and \eqref{eq:con_DDP}. $\vl_{t}$ takes the form $\vS_{t}\delta \vx_{t} + \bv_{t}$ if the necessary conditions are solved in this manner. So, the necessary conditions become:
\begin{subequations}
    \begin{equation}
        \vf_{\vx_t}\delta \vx_t + \vf_{\vu_t}\delta \vu_t - \delta \vx_{t + 1} = \bo.
        \label{eq:lin_dyn_SN}
    \end{equation}
    \begin{equation*}
        c_{\vx_t} + c_{\vx_t \vx_t}\delta \vx_t + c_{\vx_t \vu_t}\delta \vu_t 
    \end{equation*}
    \begin{equation}
        +(\vf_{\vx_t \vx_t} \delta \vx_{t} + \vf_{\vx_t \vu_{t}}\delta \vu_{t})\tran \bar{\blambda}_{t+1} + \vf_{\vx_t}\tran \vl_{t+1}
        -\vl_t = 0
        \label{eq:state_SN}
    \end{equation}
    \begin{equation*}
        c_{\vu_t} + c_{\vu_t \vu_t}\delta \vu_t + c_{\vu_t \vx_t} \delta \vx_t
    \end{equation*}
    \begin{equation}
        +(\vf_{\vu_t \vu_t} \delta \vu_{t} + \vf_{\vu_t \vx_{t}}\delta \vx_{t})\tran \bar{\blambda}_{t+1} + \vf_{u_t}\tran \vl_{t+1} = 0
        \label{eq:con_SN}
    \end{equation}    
\end{subequations}
Additionally, rather than using $\bar{\blambda}_{t}$ from the previous iteration (which is done in a typical SQP solver), it enforces linearity of $\bar{\blambda}$ in terms of $\delta\vx$ and $\delta \vu$. The backward pass on $\bar{\blambda}$ comes from eq. \eqref{eq:state_SN}. Given that $\vl_{t}^{\text{iter}-1} = \bar{\blambda}_{t}^{\text{iter}}$ and $\vl_{t+1}^{\text{iter}-1} = \bar{\blambda}_{t+1}^{\text{iter}}$ , where ``iter" denotes the current iteration number. This along with the expansions
\begin{align*}
   &c_{\vx_t}^{\text{iter}} = c_{\vx_t}^{\text{iter}-1} + c_{\vx_t \vx_t}^{\text{iter}-1}\delta \vx_t + c_{\vx_t \vu_t}^{\text{iter}-1}\delta \vu_t  + \text{H.O.T.} \\
    &\vf_{\vx_t}^{\text{iter}} = \vf_{\vx_t}^{\text{iter}-1} + \vf_{\vx_t\vx_t}^{\text{iter}-1}\delta \vx_t + \vf_{\vx_t\vu_t}^{\text{iter}-1}\delta \vu_t + \text{H.O.T.}, 
\end{align*}
 and using eq. \eqref{eq:state_SN}, one can write
\begin{equation*}
    \bar{\blambda}_{t}^{\text{iter}} = \vl_{t}^{\text{iter}-1} =  \underbrace{c_{\vx_t} + c_{\vx_t \vx_t}\delta \vx_t + c_{\vx_t \vu_t}\delta \vu_t 
        +(\vf_{\vx_t\vx_t} \delta \vx_{t} + \vf_{\vx_t \vu_{t}}\delta \vu_{t})\tran \bar{\blambda}_{t+1} + \vf_{\vx_t}\tran \vl_{t+1}}_{\text{all derivatives at (iter - 1)}},     
\end{equation*}
which under linearity assumption gives:
\begin{equation}
    \bar{\blambda}_{t}^{\text{iter}} = c_{\vx_t}^{\text{iter}} +  (\vf_{\vx_t}^{\text{iter}})\tran \bar{\blambda}_{t+1}^{\text{iter}}.
    \label{eq:back_pass_SN_lam}
\end{equation}
However, this will have errors even under linearity assumption of derivatives as the term $(\vf_{\vx_t \vx_t} \delta \vx_{t} + \vf_{\vx_t \vu_{t}}\delta \vu_{t})\tran \bar{\blambda}_{t+1}$ should have been $(\vf_{\vx_t \vx_t} \delta \vx_{t} + \vf_{\vx_t \vu_{t}}\delta \vu_{t})\tran \vl_{t+1}$ for eq.~\eqref{eq:back_pass_SN_lam} to hold even in a linear sense. But, ``Stagewise Newton" makes this approximation.
So, the backward pass for ``Stagewise Newton" becomes 
\begin{align}
    \bv_{t} = Q_{\vx}^{t} - (Q_{\vu}^{t})\tran (Q_{\vu\vu}^{t})^{-1}Q_{\vu\vx}^{t}, \\
    \vS_{t} = Q_{\vx\vx}^{t} - Q_{\vx\vu}^{t}(Q_{\vu\vu}^{t})^{-1}Q_{\vu\vx}^{t},\\
    \bar{\blambda}_{t} = c_{\vx_t} + \vf_{\vx_t}\tran \bar{\blambda}_{t+1},
    \label{eq:SN_back_pass_lam}
\end{align}
where
\begin{align*}
    &Q_{\vx}^{t} = c_{\vx_{t}} + \vf_{\vx_t}\tran \bv_{t+1} \\
    &Q_{\vu}^{t} = c_{\vu_{t}} + \vf_{\vu_t}\tran \bv_{t+1}\\
    &Q_{\vx\vx}^{t} = c_{\vx_t\vx_t} + \vf_{\vx_t}\tran \vS_{t+1}\vf_{\vx_t} + \bar{\blambda}_{t+1}\otimes \vf_{\vx_t \vx_t}\\
    &Q_{\vu\vu}^{t} = c_{\vu_t\vu_t} + \vf_{\vu_t}\tran \vS_{t+1}\vf_{\vu_t} + \bar{\blambda}_{t+1} \otimes \vf_{\vu_t \vu_t}\\
    &Q_{\vu\vx}^{t} = c_{\vu_t\vx_t} + \vf_{\vu_t}\tran \vS_{t+1}\vf_{x_t} + \bar{\blambda}_{t+1} \otimes \vf_{\vu_t \vx_t}
\end{align*}
with the terminal conditions:
\begin{align*}
    \vS_{T} = C_{T,\vx\vx}, \quad \bv_{T} = C_{T,\vx}, \quad \text{and} \quad \bar{\blambda}_{T} = C_{T,\vx}. 
\end{align*}
We reiterate that the forward pass for Stagewise Newton updates $\delta \vu_{t} = \vk_{t} + \vK_{t} \delta \hat{\vx}_{t}$ where $\delta \hat{\vx}_{t} = \vf_{\vx_{t-1}}\delta \hat{\vx}_{t-1} + \vf_{\vu_{t-1}}\delta \vu_{t-1}$, is the linearized dynamics, unlike DDP, which takes the full $\delta \vx_{t} = \vf(\bar{\vx}_{t-1} + \delta \vx_{t-1},\bar{\vu}_{t-1}+\delta \vu_{t-1}) - \vf(\bar{\vx}_{t-1},\bar{\vu}_{t-1})$.
\begin{remark}
   In reference \cite{Pantoja_SN}, the author proposes the aforementioned stagewise Newton procedure to find the Newton direction for the optimal control problem posed in this paper. A careful look at these equations clearly shows that they are not the same as the equations we have for Newton LQR above. In particular, rather than the variables $\bar{\lambda}_t$ being the multipliers from the previous iteration, they are approximated to be linear as in eq. \eqref{eq:SN_back_pass_lam}. To begin with, this is not a good approximation as explained in the aforementioned section. Along with this, not using full perturbation in the forward pass, like DDP also leads to slower convergence \cite{shoemaker1983ddp}. 
\end{remark}

We describe the iLQR algorithm next. iLQR uses the modified QP as mentioned in eq. \eqref{M-QP}. For the optimal control problem, this means that the optimization problem in eq. \eqref{eq:per_optim} becomes
\begin{dmath}
    \min_{\delta \vx_t, \delta \vu_{t}, \delta \blambda_t} \sum_{t=0}^{T-1}\Big(c_{\vx_t}\delta \vx_{t} + c_{\vu_t}\delta \vu_{t} + \frac{1}{2}\delta \vx_{t}' c_{\vx_t\vx_t} \delta \vx_t + \delta \vx_t\tran c_{\vx_t\vu_t} \delta \vu_t + \frac{1}{2} \delta \vu_{t}\tran c_{\vu_t \vu_t} \delta \vu_t  + (\bar{\blambda}_{t+1} + \delta \blambda_{t+1})\tran (\vf(\bar{\vx}_t, \bar{\vu}_t) + \vf_{\vx_t}\delta \vx_t + \vf_{\vu_t} \delta \vu_t - 
    \bar{\vx}_{t+1}  -\delta \vx_{t+1})\Big ) + C_{T, \vx}\delta \vx_T + \frac{1}{2} \delta \vx_T\tran C_{T, \vx\vx} \delta \vx_T.
    \label{eq:per_optim_iLQR}
\end{dmath}
It has the same forward and backward pass as DDP except we drop the tensor products of second order terms in dynamics. This gives:

\begin{align*}
    &Q_{\vx}^{t} = c_{\vx_{t}} + \vf_{\vx_t}\tran\bv_{t+1} \\
    &Q_{\vu}^{t} = c_{\vu_{t}} + \vf_{u_t}\tran \bv_{t+1}\\
    &Q_{\vx\vx}^{t} = c_{\vx_t\vx_t} + \vf_{\vx_t}\tran \vS_{t+1}\vf_{\vx_t} \\
    &Q_{\vu\vu}^{t} = c_{\vu_t\vu_t} + \vf_{\vu_t}\tran \vS_{t+1}\vf_{\vu_t} \\
    &Q_{\vu\vx}^{t} = c_{\vx_t\vu_t} + \vf_{\vu_t}\tran \vS_{t+1}\vf_{\vx_t} 
\end{align*}

\begin{proposition}\label{prop:ilqr_descent}
    The solution to the iLQR problem is always a descent direction for the cost $c(x,u)$ and the next solution can be made feasible via a suitable line-search procedure. In contrast, it is not necessary that the solution to Newton LQR be a descent direction. \\
\end{proposition}

This makes the tensor products drop from the $Q_{\vx\vx}^{t}, Q_{\vu\vu}^{t}$ and $Q_{\vu\vx}^{t}$ and now the net cost decrease can be guaranteed to be positive definite under the following mild assumptions.
\begin{assumption}
    $c_{\vu_t\vu_t}$ $\succ 0$. $c_{\vu_t\vu_t}$ is always positive definite for all time-steps, $t$.
\end{assumption}
\begin{assumption}
    $C_{T,\vx\vx}$ $\succeq 0$. $C_{T,\vx\vx}$ is positive semi-definite.
\end{assumption}
\begin{assumption}
    $\begin{bmatrix}
        c_{\vx_t\vx_t} && c_{\vx_t\vu_t}\\
        c_{\vu_t\vx_t} && c_{\vu_t\vu_t}
    \end{bmatrix}$ $\succeq 0$. The hessian matrix of incremental cost, $c(\vx_t,\vu_t)$, is positive semidefinite for all time-steps, $t$.
\end{assumption}
\begin{proof}

 If we have $Q_{\vu\vu}^{t}$ to be positive definite for all time-steps then we are guaranteed a reduction in cost. At the terminal, we have
\begin{equation*}
     Q_{\vu\vu}^{T-1} = c_{\vu_{T-1}\vu_{T-1}} + \vf_{\vu_{T-1}}\tran C_{T,\vx\vx}\vf_{\vu_{T-1}}
\end{equation*}
which is positive definite following from assumptions 1 and 2. 
In addition to this, the hessian of iLQR is
\begin{equation}
    \begin{bmatrix}
        Q_{\vx\vx}^{T-1} & Q_{\vx\vu}^{T-1}\\
        Q_{\vu\vx}^{T-1} & Q_{\vu\vu}^{T-1}
    \end{bmatrix} = \begin{bmatrix}
        c_{\vx_{T-1}\vx_{T-1}} &  c_{\vx_{T-1}\vu_{T-1}}\\
         c_{\vu_{T-1}\vx_{T-1}} &  c_{\vu_{T-1}\vu_{T-1}} 
    \end{bmatrix} + \begin{bmatrix}
        \vf_{\vx_{T-1}}\\
        \vf_{\vu_{T-1}}
    \end{bmatrix}\tran \begin{bmatrix}
        C_{T,\vx\vx} & C_{T,\vx\vx}\\
        C_{T,\vx\vx} & C_{T,\vx\vx}
    \end{bmatrix} \begin{bmatrix}
        \vf_{\vx_{T-1}}\\
        \vf_{\vu_{T-1}}
    \end{bmatrix}.
    \label{eq:hess_matrix}
\end{equation}
In the above equation, it is evident that the matrix on left hand side is positive definite. This follows from assumptions 2 and 3, which imply:
\begin{align*}
    \begin{bmatrix}
        Q_{\vx\vx}^{T-1} & Q_{\vx\vu}^{T-1}\\
        Q_{\vu\vx}^{T-1} & Q_{\vu\vu}^{T-1}
    \end{bmatrix} &\succeq 0, \quad \text{since} \\
    \begin{bmatrix}
        c_{\vx_{T-1}\vx_{T-1}} &  c_{\vx_{T-1}\vu_{T-1}}\\
         c_{\vu_{T-1}\vx_{T-1}} &  c_{\vu_{T-1}\vu_{T-1}}
    \end{bmatrix} &\succeq 0 \quad \text{and} \quad
    \begin{bmatrix}
         C_{T,\vx\vx} & C_{T,\vx\vx}\\ 
        C_{T,\vx\vx} & C_{T,\vx\vx}
    \end{bmatrix} \succeq 0.
\end{align*}
In addition to this, $Q_{\vu\vu}^{T-1}$ is already shown to be positive definite. So, the Schur complement \cite{Gallier2019Schur} of $\begin{bmatrix}
        Q_{\vx\vx}^{T-1} & Q_{\vx\vu}^{T-1}\\
        Q_{\vu\vx}^{T-1} & Q_{\vu\vu}^{T-1}
    \end{bmatrix}$ has to be positive semidefinite, which gives
\begin{equation}
    Q_{\vx\vx}^{T-1} - Q_{\vx\vu}^{T-1}(Q_{\vu\vu}^{T-1})^{-1}Q_{\vu\vx} \succeq 0.
\end{equation}
But the above equation is nothing but the backward pass on $\vS$:
\begin{equation*}
    \vS_{T-1} =  Q_{\vx\vx}^{T-1} - Q_{\vx\vu}^{T-1}(Q_{\vu\vu}^{T-1})^{-1}Q_{\vu\vx}^{T-1} \succeq 0
\end{equation*}
Now, one can write
\begin{equation*}
    Q_{\vu\vu}^{T-2} = c_{\vu_{T-2}\vu_{T-2}} + \vf_{\vu_{T-2}}\tran \vS_{T-1}\vf_{\vu_{T-2}}
\end{equation*}
which is always positive definite as $c_{\vu_{T-1}\vu_{T-2}}$ is positive definite following from assumption 1 and $\vS_{T-1} \succeq 0$. By induction, it follows that $Q_{\vu\vu}^{t} \succ 0$ for all time-step, $t$.
\end{proof}

Thus, this implies that the iLQR guaranties a decent direction under these mild assumptions. In contrast, DDP often requires heavy regularization.





\section{Contrasting the algorithmic performance}\label{sec:algorithms}

While the preceding discussion outlined the theoretical aspects of iLQR, DDP, and Stagewise-Newton, it is equally important to understand how these algorithms behave in practice. In this section, we present a direct comparison of their performance on representative control tasks, explicitly focusing on the unregularized forms of both iLQR and DDP. By analyzing convergence rates, predicted versus actual cost reductions, and step sizes, we aim to reveal the practical advantages and shortcomings that arise solely from their respective local modeling choices, independent of any stabilization or damping heuristics. We consider six problems to show this comparison. We have pendulum and cartpole swing-up tasks. A hypersensitive problem \cite{Rao2000} with dynamics given by:
\begin{equation*}
    \dot{\vx} = -\vx^{3} + \vu
\end{equation*}
is considered and referred to as ``Test Problem-1". Another 1D test problem was taken from \cite{jacobson1970ddp}. The governing equation for this problem is
\begin{equation*}
    \dot{\vx} = -0.2\vx + 10\tanh(\vu).
\end{equation*}
This problem is labeled as ``Test Problem -2". 
We consider the incremental as well as the terminal cost to be quadratic for all the problems. The general form of cost is given by:
\begin{align*}
    J = \min_{\vu(t)} \frac{1}{2}\int_{0}^{t_f} (\vx-\vx_f)\tran \vQ(\vx - \vx_f) + \vu\tran \vR\vu ~~~dt \\
    + \frac{1}{2} (\vx_{t_f} - \vx_f)\tran \vQ_f (\vx_{t_f} - \vx_f)
\end{align*}
where the parameters of each problem is summarized in Table \ref{tab:sys_params}. The term $dt$ denotes the size of step used for Euler integration and $\Delta t$ denotes the length of zero order control hold used to solve the problem.

\begin{table*}[!htpb]
    \centering
    \tiny
    \begin{tabular}{|c|c|c|c|c|c|c|c|c|c|}
        \hline
        System & $ t_f$(s) & dt(s) & $\Delta t$(s) & T & $\vQ$ & $\vR$ & $\vQ_f$ & $\vx_{0}$ & $\vx_{f}$\\
        \hline
        Pendulum    & 5      & $10^{-3}$ &  0.1 & 50 & 3$I_{2\times 2}$ &  3 & 30$I_{2\times 2}$  & $[0, 0]'$ & $[\pi,0]'$     \\
        \hline
        Cartpole    & 3      & $10^{-3}$ &  0.1 & 30 & 10$I_{4\times 4}$ & 0.001 & 10000$I_{4\times 4}$ &  $[0,0,\pi,0]'$ & $[0,0,0,0]'$  \\
        \hline
        TP-1    & 25      & $10^{-5}$ &  0.01 & 2500 & 2 & 2 & $10^{10}$ & 1 & 1.5    \\
        \hline
        TP-2 & 0.5 & $10^{-4}$ &  0.01 & 50 & 20 & 2 & 20 & 5 & 0\\
        \hline
    \end{tabular}
    \caption{Parameters for all four problems. TP-1 denotes test problem 1 and TP-2 denotes test problem 2.}
    \label{tab:sys_params}
\end{table*}
In addition to these problems, we also borrowed two problems from \cite{LiaoShoemaker1992}. We label the first problem as ``Test-Problem 3." It is described as:
\begin{align*}
    \min \mathcal{J} =& \sum_{t=1}^{T-1}\left( \sum_{i=1}^{n} \left( (\vx_{t})_{i} + \frac{1}{4} \right)^{4} + \sum_{j=1}^{m} \left( (\vu_{t})_{j} + \frac{1}{4} \right)^{4} \right) + \sum_{i=1}^{n} \left( (\vx_{T})_{i} + \frac{1}{4} \right)^{4},\\
    &\text{where}~~ \vx_{t+1} = \mathbf{A}\vx_{t} + \mathbf{B}\vu_{t} + (\vx_{t}\tran \mathbf{C} \vu_{t}) \boldsymbol{\gamma} \quad t = 1,2,\cdots T-1\\
    &\vx_{1} = \boldsymbol{0}_{n\times 1} \in \mathbb{R}^{n}\\
    &\mathbf{A} \in \mathbb{R}^{n \times n}~~\text{and}~~(\mathbf{A})_{ij} = \begin{cases}
        0.5 \quad \text{if}~i=j\\
        0.25 \quad \text{if}~j=i+1\\
        -0.25 \quad \text{if}~j=i-1\\
        0\quad \text{elsewhere}
    \end{cases}~~i=1,\cdots n\\
    & \mathbf{B} \in \mathbb{R}^{n \times m}, ~~(\mathbf{B})_{ij} = \frac{i-j}{m+n},~~i=1,\cdots n ~~\text{and}~~j= 1,\cdots m\\
    & \mathbf{C} \in \mathbb{R}^{n \times m}, ~~(\mathbf{C})_{ij} = \mu\frac{(i+j)}{m+n},~\mu~\text{is a scalar},\\
    & \boldsymbol{\gamma} \in \mathbb{R}^{n}~\text{ and }~ \boldsymbol{\gamma} = \boldsymbol{1}_{n \times 1}.
\end{align*}
In the problem described above, $(\cdot)_{ij}$ denote the element of a matrix at $i$-th row and $j$-th column.

The final problem is labeled ``Test-Problem 4" and is described below:
\begin{align*}
    \min \mathcal{J} =& \sum_{t=1}^{T-1} || \vx_{t} ||_{2}^{2} \left( \sin \left( \frac{||\vu_{t}||_{2}^{2}}{m} \right) + 1 \right)  + || \vx_{T}||_{2}^{2}\\
    &\text{where}~(\mathbf{x}_{t+1})_{i} = \sin \left( (\mathbf{x}_{t})_{i} \right) + \left( \mathbf{F} \mathbf{W}(\mathbf{u}_{t}) \right)_{i}\\
    &(\cdot)_{i}~\text{denotes the }~i-\text{th component of the vector}~(\cdot)\\
    &(\mathbf{x}_{1})_{i} = \frac{i}{2n}\quad i = 1,\cdots n\\
    &(\mathbf{F})_{ij} = \frac{i+j}{2n}\quad i = 1, \cdots n~~\text{and}~~j=1,\cdots m\\
    &\mathbf{W}(\mathbf{u}_{t}) = \left( \sin\left(\left(\mathbf{u}_{t}\right)_{1}\right),\cdots,\sin\left(\left(\mathbf{u}_{t}\right)_{m}\right)\right)\tran \in \mathbb{R}^{m}\\
    &\vx_{t} \in \mathbb{R}^{n} \quad \text{and} \quad \vu_{t} \in \mathbb{R}^{m}.
\end{align*}


\subsection{The impact of Forward-Pass}

A critical distinction between the Stagewise-Newton method and the DDP/iLQR family of algorithms lies in the forward pass, specifically in how state perturbations are propagated to compute the control update. In DDP and iLQR, the feedback term depends on the full nonlinear state perturbation, $\delta \vx_t$. The control update is
$
\delta \vu_t = \alpha \vk_t + \vK_{t} \delta \vx_t
$,
where $\delta \vx_t$ is the difference between the new and previous (nominal) trajectories:
\begin{align*}
    \delta \vx_t = \vf(\vx_{t}, \vu_{t}) - \vf(\bar{\vx}_{t}, \bar{\vu}_{t}).
\end{align*}
In contrast, the Stagewise-Newton method employs a control update, $\delta \vu_t = \alpha \vk_t + \vK_{t} \delta \hat{\vx}_t$, that relies on a linearized propagation of the state perturbation, $\delta \hat{\vx}_t$:
\begin{align*}
    \delta \hat{\vx}_{t+1} = \vf_{\vx_t}\delta \hat{\vx}_{t} + \vf_{\vu_t} \delta \vu_{t},\\
    \text{where}~~\delta \hat{\vx}_{t} = \vf_{\vx_{t-1}}\delta \hat{\vx}_{t-1} + \vf_{\vu_{t-1}} \delta \vu_{t-1}.
\end{align*}
The performance implications of this choice were highlighted by Liao and Shoemaker \cite{LiaoShoemaker1992}, who observed that the linearized forward pass of the Stagewise-Newton method leads to significantly slower convergence compared to DDP. To isolate the effect of the forward pass, they introduced a ``Mixed" algorithm combining a DDP-style forward pass with a Stagewise-Newton backward pass, confirming that using the full nonlinear state perturbation, $\delta \vx_t$, accelerates convergence.

To analyze this difference quantitatively, we examine the Taylor series expansion of the state perturbation. Since the initial state is fixed, $\delta \vx_0 = \boldsymbol{0}$, the first control perturbation is purely feedforward: $\delta \vu_0 = \alpha \vk_0$. The resulting state perturbation at the next step, $\delta \vx_1$, can be expanded as:
\begin{dmath*}
    \delta \vx_{1} = \vf_{\vx_{0}}\delta \vx_{0} + \vf_{\vu_{0}}\delta \vu_{0} + \frac{1}{2} \delta \vx_{0}\tran \vf_{\vx_{0}\vx_{0}} \delta \vx_{0}+ \delta \vu_{0}\tran \vf_{\vu_{0}\vx_{0}} \delta \vx_{0}+ \frac{1}{2} \delta \vu_{0}\tran \vf_{\vu_{0}\vu_{0}} \delta \vu_{0} + \alpha^{3}(\text{third order terms}) + \alpha^{4}(\text{fourth order terms}) + \cdots \cdots
\end{dmath*}
Substituting the initial conditions yields an expansion in powers of the line search parameter $\alpha$:
\begin{dmath*}
    \delta \vx_{1} = \alpha \vf_{\vu_{0}} \vk_{0} + \frac{1}{2} \alpha^{2} \vk_{0}\tran \vf_{\vu_{0}\vu_{0}} \vk_{0} + + \alpha^{3}(\text{third order terms}) + \alpha^{4}(\text{fourth order terms}) + \cdots \cdots
\end{dmath*}
This recursive structure propagates through time. The state perturbation $\delta \vx_t$ at any time step $t$ is an expansion in powers of $\alpha$, where the coefficients depend on the system Jacobians and the feedforward control modifications $\{\vk_i\}_{i=0}^{t-1}$. The general form of the first-order term in $\alpha$ is:
\begin{dmath*}
    \delta \vx_{t} = \alpha \left[ \vf_{\vu_{t-1}}\vk_{t-1} + \sum_{j=0}^{t-2} \left( \prod_{i=j+1}^{t-1} (\vf_{\vx_{i}} + \vf_{\vu_{i}}\vK_{i}) \right) \vf_{\vu_{j}}\vk_{j} \right] +  \alpha^{2}(\text{second order terms}) + \alpha^{3}(\text{third order terms}) + \cdots \cdots
\end{dmath*}

\begin{remark}
In the perturbation expansion of the dynamics, each $k$-th order term is scaled by a factor of $\alpha^k$. Consequently, one can diminish the influence of higher-order terms and recover the linearized dynamics by reducing $\alpha$.
\end{remark}

The Stagewise-Newton method, by design, retains only the first-order terms in $\alpha$ from this expansion for its feedback correction. This corresponds to an implicit assumption that higher-order dynamics are negligible, which can be a poor approximation. In contrast, DDP and iLQR use a line search on $\alpha$ to actively manage the influence of these higher-order terms, reducing $\alpha$ only when necessary to ensure a sufficient decrease in the cost function. This adaptive approach avoids prematurely discarding valuable nonlinear information.

This theoretical difference is demonstrated empirically in Fig.~\ref{fig:alpha_comp_SN} for the pendulum swing-up task. The figure compares four methods of propagating the state perturbation $\delta\theta$: 
1) the ground truth (full nonlinear difference), 2) a local linear approximation, 3) a local quadratic approximation, and 4) the recursive linear propagation used in Stagewise-Newton. For a given iteration, the local linear and quadratic models are nearly indistinguishable from the ground truth, indicating the dynamics are well-behaved locally. However, the Stagewise-Newton propagation (``Linear Prop.") significantly deviates from the ground truth, especially for larger $\alpha$. This error reduces as $\alpha$ is reduced (e.g., $\alpha=0.01$), confirming that the method's validity is restricted to small step sizes.

These findings on the forward pass directly correlate with convergence speed, as shown in Table~\ref{tab:iterations_to_optimal}. The data, including results for DDP and Newton from \cite{LiaoShoemaker1992}, demonstrates that both iLQR and DDP significantly outperform the Newton method due to their handling of nonlinearity. Notably, iLQR converges faster than DDP in this instance. The reasons for this, related to the backward pass and regularization, will be explored in the next section.

\begin{figure}[!htbp]
    \centering
    \subfloat[$\alpha=1$]{\includegraphics[width=0.48\linewidth]{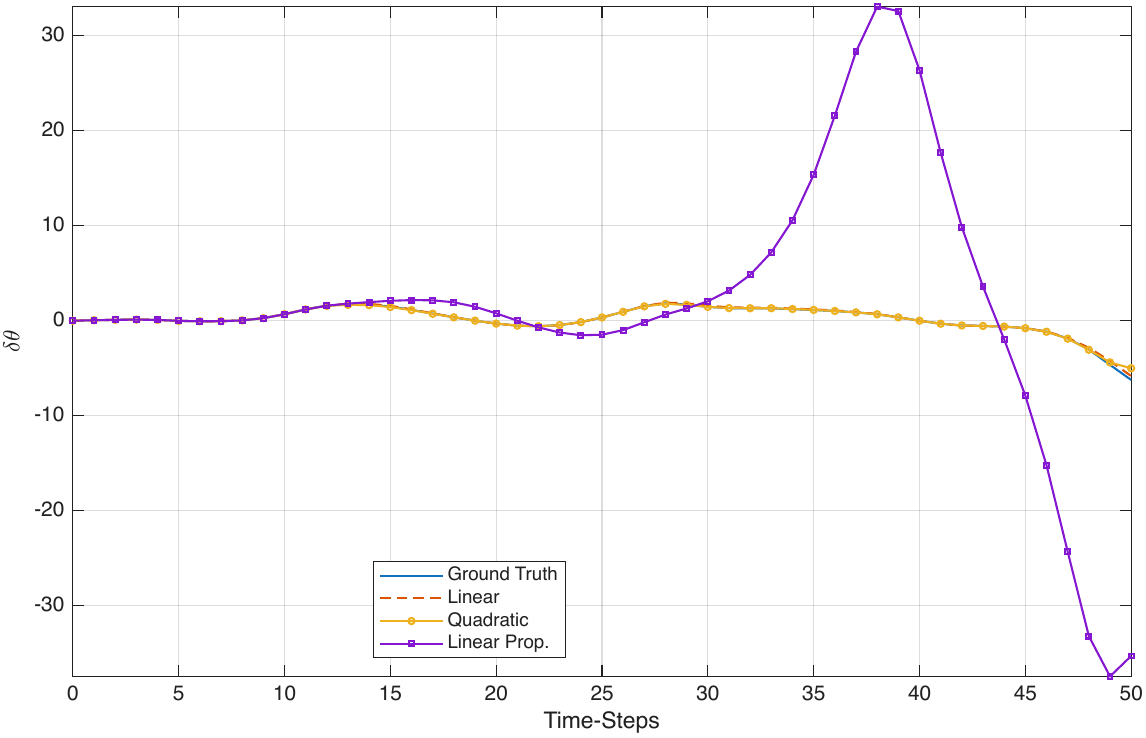}\label{subfig:alpha1}}
   \hfill
   \subfloat[$\alpha=0.1$]{\includegraphics[width=0.48\linewidth]{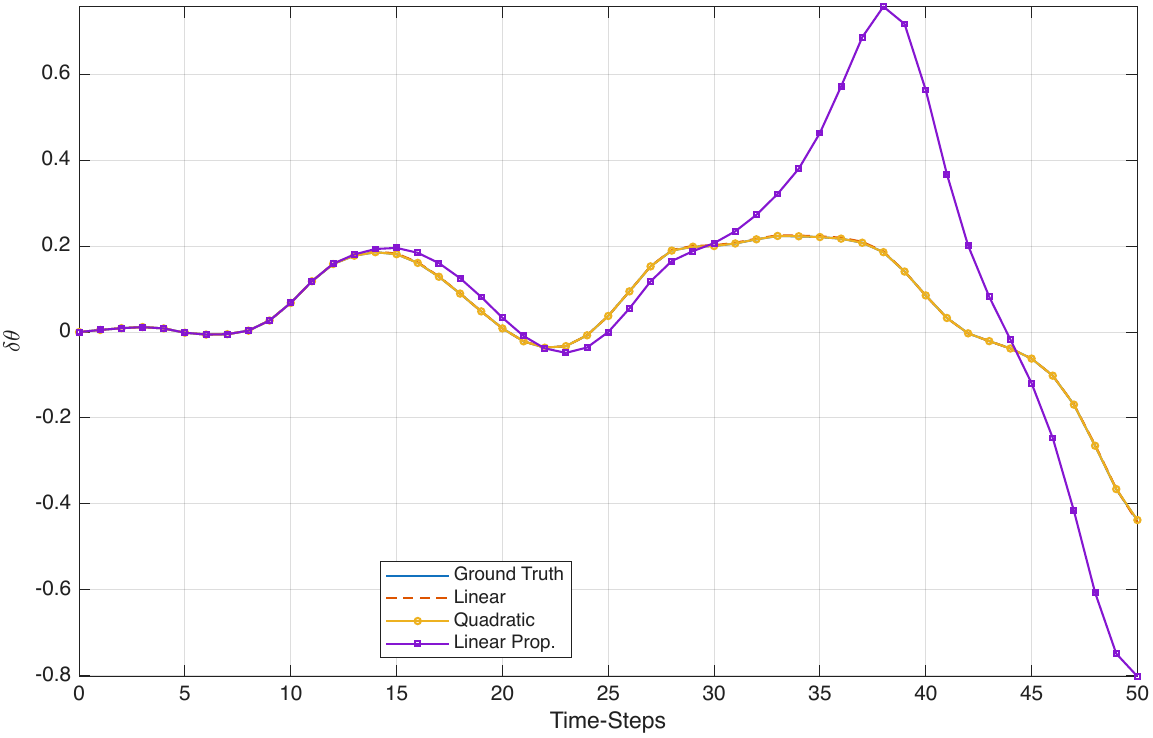}\label{subfig:alpha01}}%
   \vfill
   \subfloat[$\alpha=0.05$]{\includegraphics[width=0.48\linewidth]{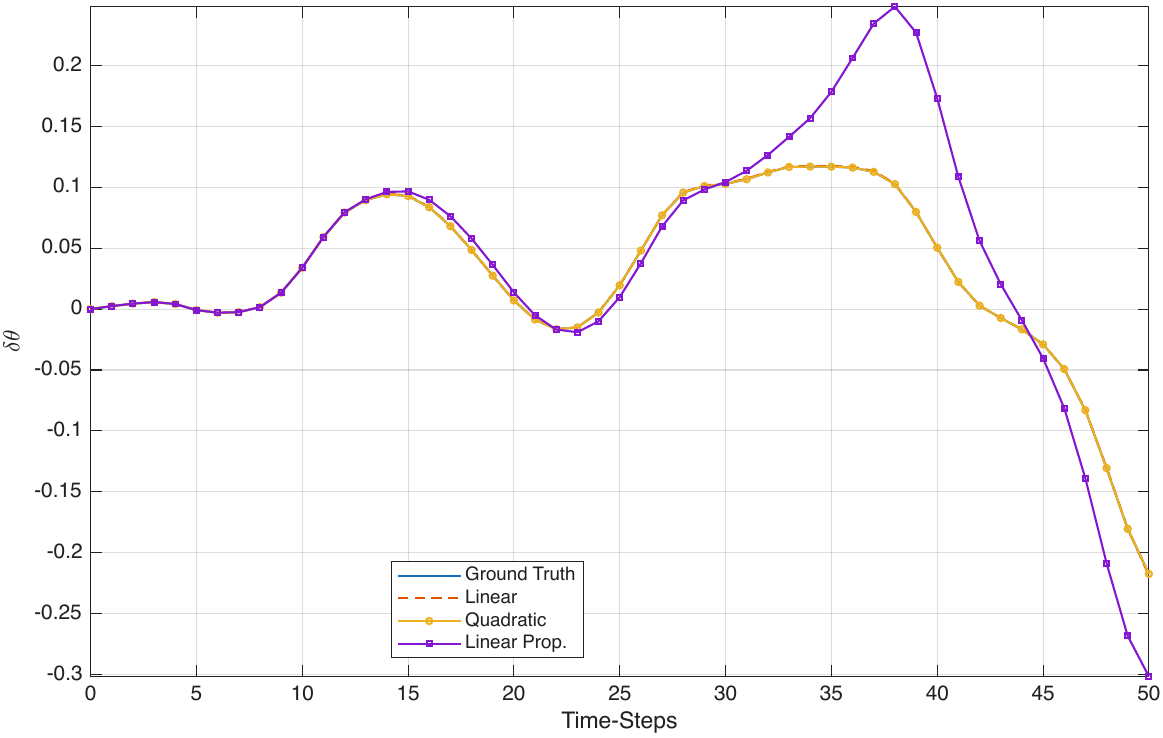}\label{subfig:alpha005}}
   \hfill
   \subfloat[$\alpha=0.01$]{\includegraphics[width=0.48\linewidth]{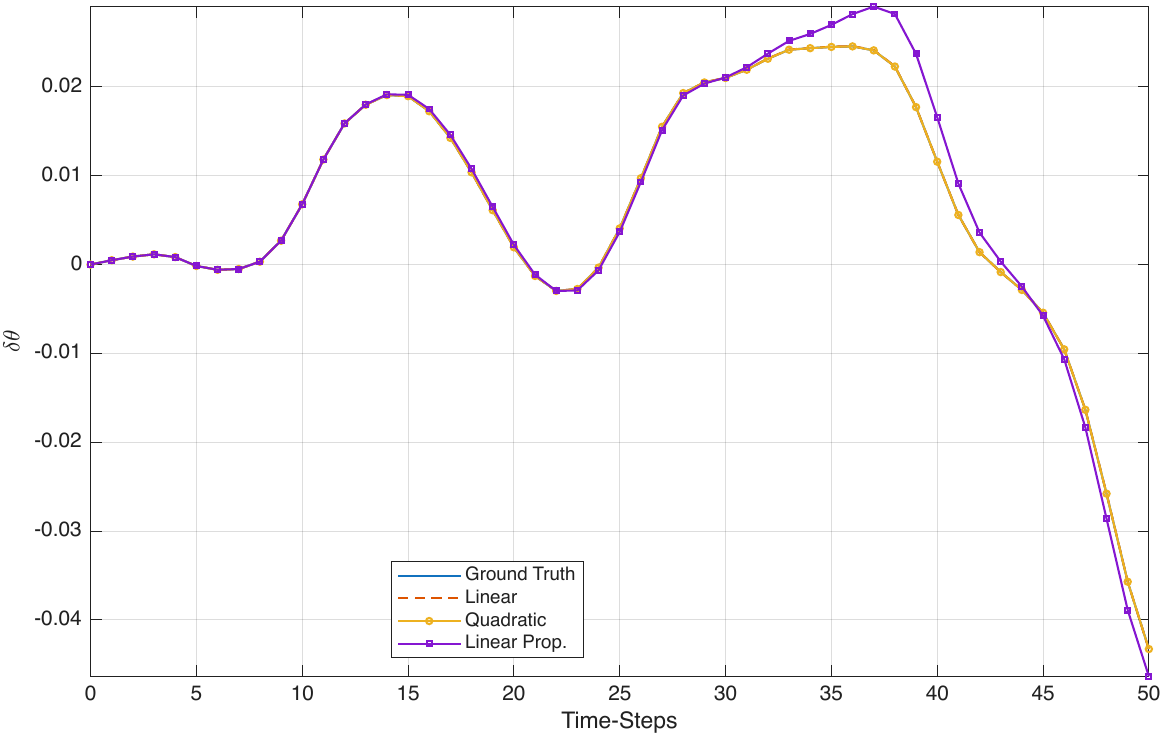}\label{subfig:alpha001}}
    \caption{Propagation of the state perturbation, $\delta\theta$, for the pendulum task. The ground truth (solid line) is compared with predictions from a local linear model (dashed), a local quadratic model (circles), and the recursive Stagewise-Newton update (squares). The Stagewise-Newton update only converges to the ground truth for very small line search parameters ($\alpha \ll 1$), whereas the local models remain accurate, highlighting the importance of using the full state perturbation in the forward pass.}
    \label{fig:alpha_comp_SN}
\end{figure}

\begin{table}[!htbp]
    \centering
    \caption{Number of iterations to convergence for Test Problem 4. $T=100$, $n=100$, and $m=10$.}
    \label{tab:iterations_to_optimal}
    \begin{tabular}{|l|c|c|c|}
        \hline
        & \multicolumn{3}{c|}{Number of Iterations} \\
        \cline{2-4}
        Starting Point\textsuperscript{a} & iLQR & DDP\textsuperscript{b} & Newton\textsuperscript{b}  \\
        \hline
        1 & 4 & 8 & 132\\
        2 & 4 & 9 & 127\\
        3 & 5 & 10 & 128\\
        4 & 4 & 10 & 134\\
        5 & 5 & 10 & 134\\
        \hline
        \multicolumn{4}{l}{\textsuperscript{a}\footnotesize{Initial control guesses are defined in Table~\ref{tab:starting_points}.}} \\
        \multicolumn{4}{l}{\textsuperscript{b}\footnotesize{Data for DDP and Newton methods from \cite{LiaoShoemaker1992}.}}
    \end{tabular}
\end{table}

\begin{table}[htbp]
    \centering
    \caption{Initial control trajectories for Test Problem 4 ($T=100, m=10$).}
    \label{tab:starting_points}
    \begin{tabular}{|c|l|}
        \hline
        \textbf{Starting Point} & \textbf{Initial Control}, $\vu_t$ \\
        \hline
        1 & $\vu_{t} = \boldsymbol{0}$\\
        \hline
        2 & $\vu_{t} = 0.01 \cdot \boldsymbol{1} $\\
        \hline
        3 & $\vu_t = -0.01 \cdot \boldsymbol{1} $ \\
        \hline
        4 &
        $\begin{cases}
          \vu_{t} = 0.01 \cdot \boldsymbol{1} & t \text{ is odd} \\
          \vu_{t} = \boldsymbol{0}  & t \text{ is even}
        \end{cases}
        $ \\
        \hline
        5 &
        $\begin{cases}
          \vu_{t} = -0.01 \cdot \boldsymbol{1} & t \text{ is odd} \\
          \vu_{t} = \boldsymbol{0}   & t \text{ is even}
        \end{cases}
        $ \\
        \hline
    \end{tabular}
\end{table}

\subsection{The need for regularization in DDP}
In this subsection, we examine the descent directions predicted by iLQR and DDP. In other words, we present the expected reduction in cost predicted at different time steps of iLQR and DDP. The net ``expected cost reduction" in these algorithms is given by eq.~\eqref{eq:exp_cost_red}. It should be noted that the term $Q_{\vu\vu}^{t}$ should be positive definite to guarantee a certain reduction in cost. From Proposition \ref{prop:ilqr_descent}, it can stated that under some mild assumptions, this term is always positive definite for iLQR as described in the previous section. However, it may become negative-definite/indefinite for DDP. If we do a similar analysis like iLQR ( (eq.~\eqref{eq:hess_matrix} in Proposition \ref{prop:ilqr_descent}) for DDP, we get
\begin{align*}
    \begin{bmatrix}
        Q_{\vx\vx}^{T-1} & Q_{\vx\vu}^{T-1}\\
        Q_{\vu\vx}^{T-1} & Q_{\vu\vu}^{T-1}
    \end{bmatrix} = 
    \begin{bmatrix}
        c_{\vx_{T-1}\vx_{T-1}} &  c_{\vx_{T-1}\vu_{T-1}}\\
         c_{\vu_{T-1}\vx_{T-1}} &  c_{\vu_{T-1}\vu_{T-1}} 
    \end{bmatrix} +& \begin{bmatrix}
        \vf_{\vx_{T-1}}\\
        \vf_{\vu_{T-1}}
    \end{bmatrix}\tran \begin{bmatrix}
        C_{T,\vx\vx} & C_{T,\vx\vx}\\
        C_{T,\vx\vx} & C_{T,\vx\vx}
    \end{bmatrix} \begin{bmatrix}
        \vf_{\vx_{T-1}}\\
        \vf_{\vu_{T-1}}
    \end{bmatrix}\\
    &+ \blambda_{T} \otimes 
    \begin{bmatrix}
        \vf_{\vx_{T-1}\vx_{T-1}} & \vf_{\vx_{T-1}\vu_{T-1}}\\
        \vf_{\vu_{T-1}\vx_{T-1}} & \vf_{\vu_{T-1}\vu_{T-1}}
    \end{bmatrix}.
\end{align*}
The guarantee of a descent direction is lost due to the final term. Unlike the other components, this term involving the dynamics' Hessian is generally indefinite and depends on the specific trajectory. Its presence prevents any guarantee on the positive definiteness of $\mathbf{Q}_{\vu\vu}^t$ without imposing strong, and often impractical, assumptions on the system dynamics. Therefore, DDP must employ heavy regularization to enforce this condition and ensure a valid descent direction.

\begin{figure}[!htbp]
    \centering
    \subfloat[Cart-pole swing-up.\label{subfig:cartpole}]{\includegraphics[width=0.48\linewidth]{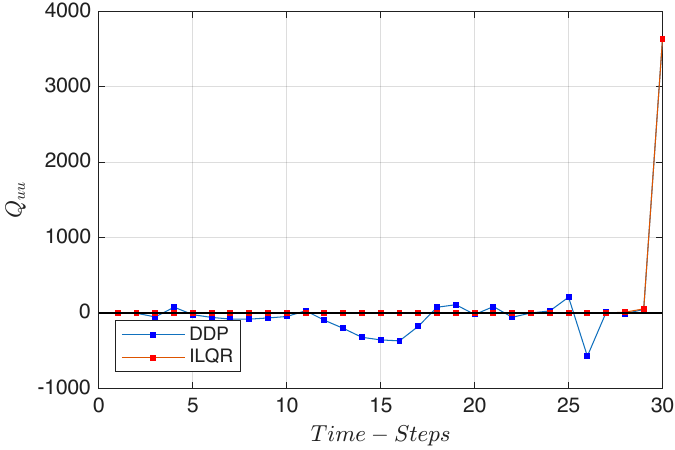}}
   \hfill
   \subfloat[Pendulum swing-up.\label{subfig:pendulum}] {\includegraphics[width=0.48\linewidth]{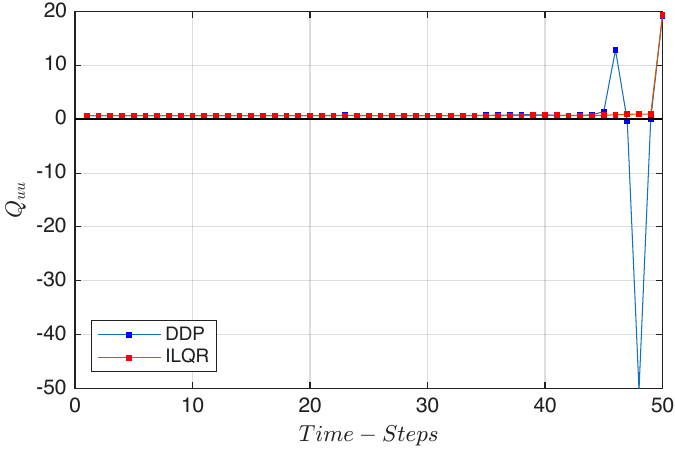}}%
   \hfill
   \subfloat[Test Problem-1.\label{subfig:TP1}] {\includegraphics[width=0.48\linewidth]{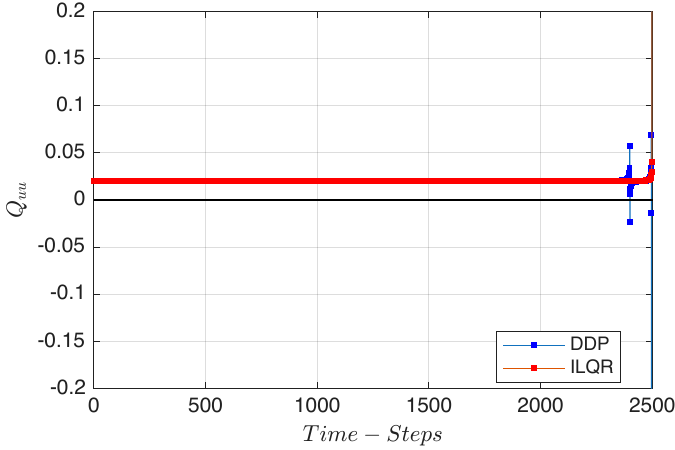}}
   \hfill
   \subfloat[Test Problem-2.\label{subfig:TP2}] {\includegraphics[width=0.48\linewidth]{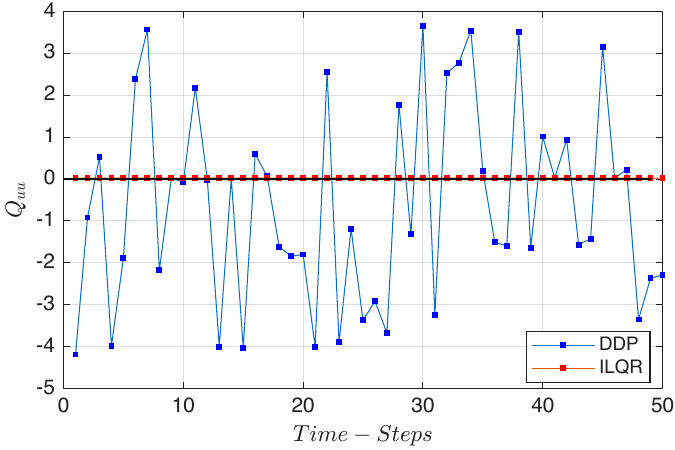}}
    \caption{The plot of $Q_{\vu\vu}$ for all the problems for the first iteration starting with random initial guess.}
    \label{fig:Quu_pd}
\end{figure}

Figure \ref{fig:Quu_pd} illustrates that $Q_{\vu\vu}$ remains positive in all the problems for iLQR. Since there is only a single control input, $Q_{\vu\vu}$ is scalar in all these cases. However, in DDP, it becomes negative at certain points, which can lead to an increase in cost at some time steps. This mixture of positive and negative values for $Q_{\vu\vu}$ across the trajectory necessitates strong regularization, as the predicted change in cost may otherwise increase. For the pendulum case shown in Fig.~\ref{subfig:pendulum}, the predicted cost variations are $\Delta J_{\text{DDP}} = -224.0154$ and $\Delta J_{\text{iLQR}} = -272.9982$. Both methods predict a cost decrease, with iLQR predicting a better decrease in this unregularized setting since $Q_{\vu\vu}$ remains positive throughout. In contrast, Fig.~\ref{subfig:cartpole} presents the $Q_{\vu\vu}$ profile for the cart-pole. Here, DDP yields $\Delta J_{\text{DDP}} = 3.0172 \times 10^7$, indicating a cost increase, while iLQR predicts $\Delta J_{\text{iLQR}} = -5.5692 \times 10^5$, reflecting a cost reduction. For Test Problem-1, $\Delta J_{\text{DDP}} = -4.2010\times 10^7$ and $\Delta J_{\text{iLQR}} = -4.3949 \times 10^7$. For Test Problem-2, $\Delta J_{\text{DDP}} = -120.19$ and $\Delta J_{\text{iLQR}} = -403.24$.

To compare the performance of the algorithms under regularization, we analyze their application to ``Test Problem-4'', a benchmark known for its high degree of nonlinearity. The convergence results are summarized in Table~\ref{tab:iterations_to_optimal}. The challenging nature of this problem necessitated the use of regularization to ensure convergence for both DDP and iLQR from all tested initial conditions.

A significant distinction emerged in the regularization strategy required by each algorithm. For iLQR, a minimal, non-problem-specific regularization was sufficient for robust convergence. We employed a simple Levenberg-Marquardt style modification, where $\mathbf{Q}_{\vu\vu}$ is shifted by $\beta\mathbf{I}$ only if it is not positive definite, with $\beta$ chosen to ensure positive definiteness. In contrast, DDP required the more specialized regularization (``Shift schemes'') described by Liao and Shoemaker~\cite{LiaoShoemaker1992} to achieve stability.

Despite the simplicity of its regularization, iLQR demonstrated superior performance, consistently converging in 4--5 iterations across all starting points. DDP, even with its tailored regularization scheme, required 8--10 iterations. This result highlights a key practical advantage of iLQR: its inherent stability, derived from the Gauss-Newton approximation, allows it to achieve rapid and robust convergence with minimal intervention, whereas DDP's reliance on the full Hessian often demands more complex regularization strategies on challenging problems.

In addition to this, ``Test Problem-3" satisfies all the assumptions for iLQR to guarantee a descent direction. To analyze this, we compare iLQR, DDP and Stagewise-Newton for this problem in Table~\ref{tab:test_problem_3_results_alt}. iLQR converged in fewer iterations compared to DDP or Stagewise-Newton without needing any regulalrization.

\begin{table}[htbp]
    \centering
    \caption{Numerical results for Test Problem-3 with $n=100$, $m=50$, $T=20$, and $\mu = 1/20$.}
    \label{tab:test_problem_3_results_alt}
    \begin{tabular}{|l|c|c|c|}
        \hline
        & iLQR & DDP\textsuperscript{b} & Newton\textsuperscript{b} \\
        \hline
        \# of iter.         & 6        & 9       & 14 \\
        optimal obj. value($\mathcal{J}^{*}$)  & 58.32139 & 58.32138 & 58.32139 \\
        Regularization     & No    & Yes    & Yes \\
        \hline
        \multicolumn{4}{l}{\textsuperscript{b}\footnotesize{Data for DDP and Newton methods from \cite{LiaoShoemaker1992}.}}
    \end{tabular}
\end{table}

\subsection{Cooling of the line-search parameter}
Building on the previous observation that iLQR converges more efficiently than DDP, this section investigates a key contributing factor: the quality of the descent direction computed by DDP. Far from the optimal trajectory, the inclusion of the full dynamics Hessian can degrade the DDP descent direction. This degradation has two primary consequences. First, it provides an explanation for the superior convergence speed of iLQR observed previously, even in regularized settings. Second, it often forces the line search to select conservatively small step sizes, thereby hindering overall progress.

The observed slower convergence of DDP can be traced to a methodological inconsistency within the algorithm. The backward pass minimizes a quadratic cost model, but to achieve local quadratic convergence, it asymmetrically applies a second-order dynamic model to the linear value term $\bv_t$ and a first-order model to the quadratic term $\vS_t$ \cite{shoemaker1983ddp}. This creates a descent direction based on a composite, approximated state perturbation. However, the forward pass uses the true nonlinear perturbation. This discrepancy between the model used to generate the search direction and the model used to evaluate it can degrade the direction's quality. As a result, the algorithm's progress is hindered, manifesting as both slower overall convergence and a frequent need to reduce the line search parameter $\alpha$.
The quadratic convergence can be written mathematically as:
\begin{align}
    \left|\left| \{ \vu_{t}^{\text{iter}+1}\}_{t=0}^{T-1} -\{ \vu_{t}^{*}\}_{t=0}^{T-1}\right|\right|_{2} \leq c \left|\left| \{ \vu_{t}^{\text{iter}}\}_{t=0}^{T-1} -\{ \vu_{t}^{*}\}_{t=0}^{T-1}\right|\right|_{2}^{2},
    \label{eq:con_map}
\end{align}
where iteration denotes the current iteration number, $c$ is a constant and $\{ \vu_{t}^{*}\}_{t=0}^{T-1}$ denotes the optimal control trajectory.

In this subsection, we show that the corrupt descent direction in DDP leads to reduction in line search parameter and is one of the reasons for small values which makes $c$ in eq.~\eqref{eq:con_map}. In this work, the line search parameter, $\alpha$ is reduced, if the ratio $\frac{\Delta J_{\text{actual}}}{\Delta J_{\text{pred}}}<0$. It is a common belief that DDP accepts larger steps than iLQR; owing to its quadratic convergence. Equivalently, for a given decrease in cost, the line-search parameter~$\alpha$ is often higher under DDP than under iLQR, motivated by the idea that DDP forms a better local model by including a second-order expansion of the dynamics. In this section, we show empirical evidence that DDP can lead to excessive cooling of the line-search parameter, $\alpha$, and hence, can take a much smaller step than iLQR.

\begin{figure}[!htbp]
    \centering
    \subfloat[Learning rate vs no. of iterations.\label{subfig:alpha_pendulum}]{\includegraphics[width=.48\linewidth]{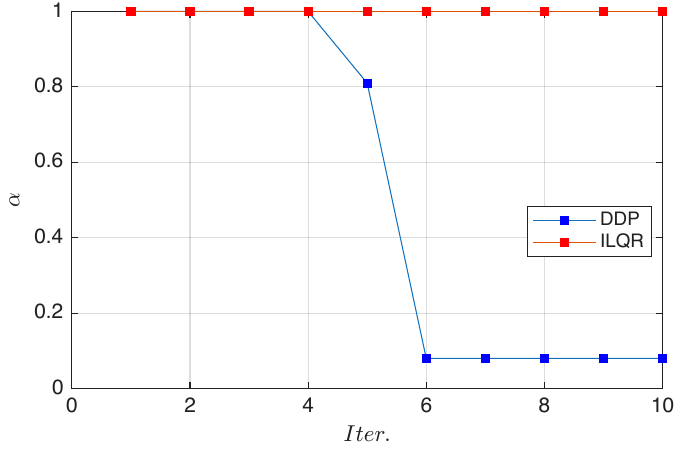}}
   \hfill
   \subfloat[Cost vs no. of iterations.\label{subfig:cost_pendulum}] {\includegraphics[width=0.48\linewidth]{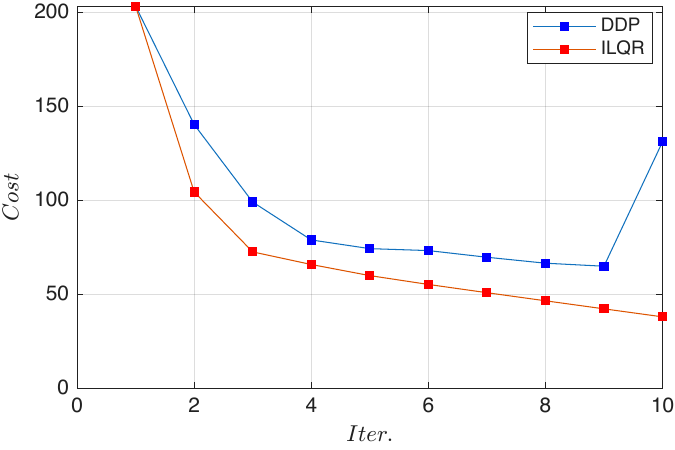}}%
    \caption{The plot of learning rate $\alpha$ and cost vs iteration for the pendulum swing-up task starting with random initial guesses. DDP cools down $\alpha$ as the cost predictions are corrupted.}
    \label{fig:cool_pendulum}
\end{figure}

\begin{figure}[!htbp]
    \centering
    \subfloat[Learning rate vs no. of iterations.\label{subfig:alpha_cartpole}]{\includegraphics[width=.48\linewidth]{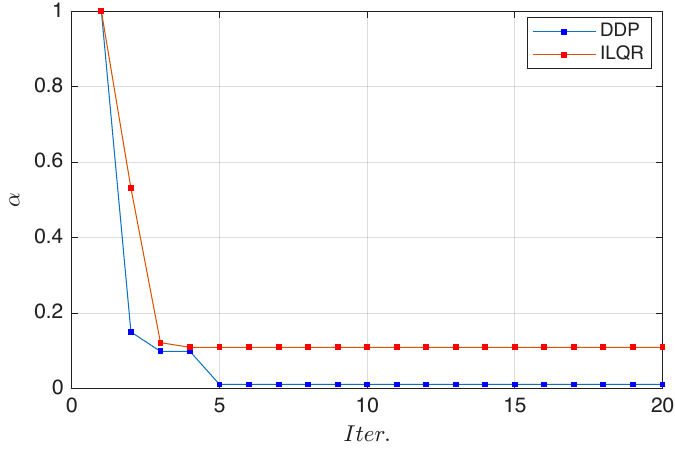}}
   \hfill
   \subfloat[Cost vs no. of iterations.\label{subfig:cost_cartpole}] {\includegraphics[width=0.48\linewidth]{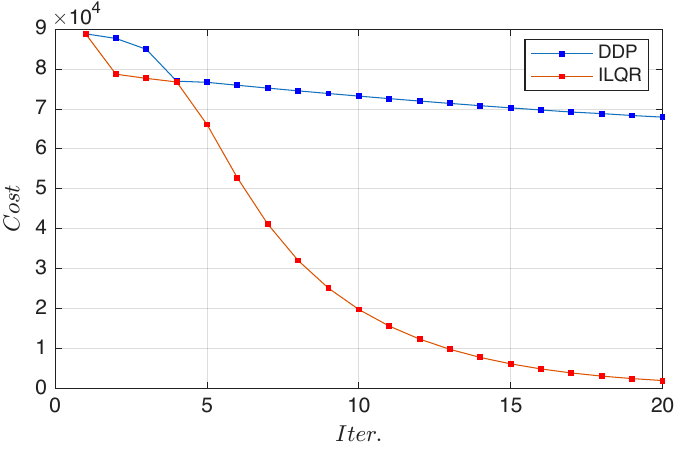}}%
    \caption{The plot of learning rate $\alpha$ and cost vs iteration for the cart-pole problem starting with random initial guesses. DDP cools down $\alpha$ as the cost change predictions are corrupted.}
    \label{fig:cool_cartpole}
\end{figure}

\begin{figure}[!htbp]
    \centering
    \subfloat[Learning rate vs no. of iterations.\label{subfig:alpha_TP1}]{\includegraphics[width=.48\linewidth]{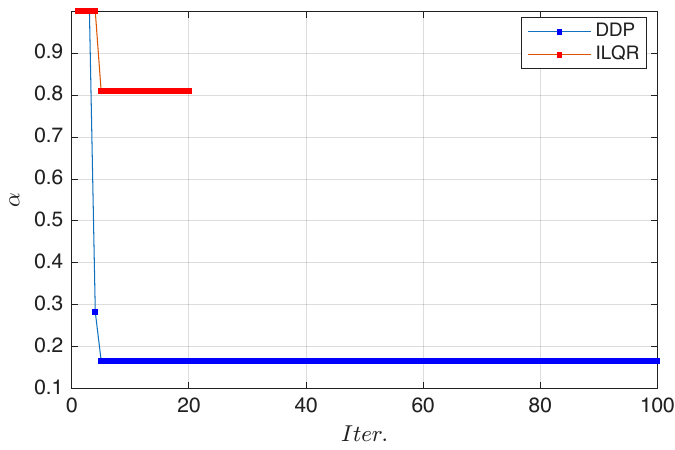}}
   \hfill
   \subfloat[Cost vs no. of iterations.\label{subfig:cost_TP1}] {\includegraphics[width=0.48\linewidth]{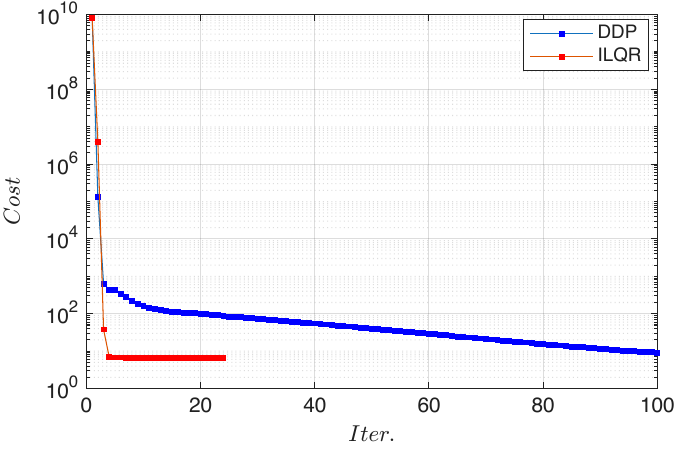}}%
    \caption{The plot of learning rate $\alpha$ and cost vs iteration for Test Problem-1 starting with random initial guess. iLQR converges in 24 iterations.}
    \label{fig:cool_TP1}
\end{figure}

\begin{figure}[!htbp]
    \centering
    \subfloat[Learning rate vs no. of iterations.\label{subfig:alpha_TP2}]{\includegraphics[width=.48\linewidth]{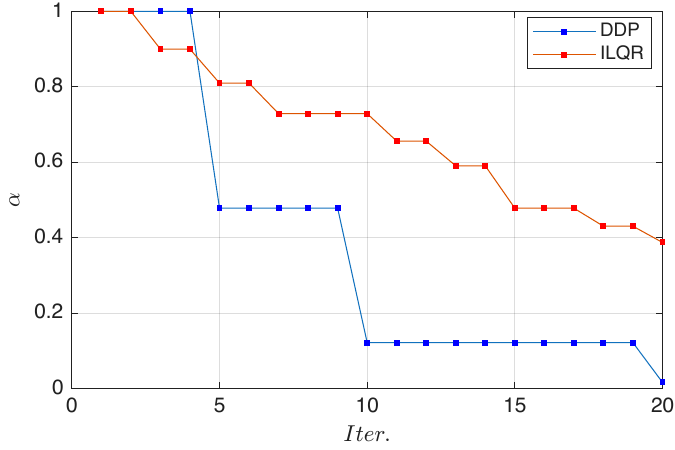}}
   \hfill
   \subfloat[Cost vs no. of iterations.\label{subfig:cost_TP2}] {\includegraphics[width=0.48\linewidth]{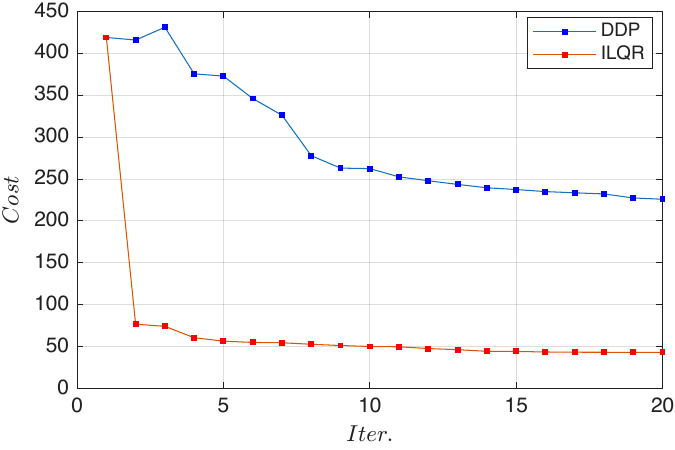}}%
    \caption{The plot of learning rate $\alpha$ and cost vs iteration for Test Problem-2 starting with random initial guess. DDP chooses bigger $\alpha$ for iterations 4 and 5 but cools down $\alpha$ below iLQR eventually.}
    \label{fig:cool_TP2}
\end{figure}

iLQR ties both prediction and realization to a single (first-order) state-propagation model, which bounds the necessary backtracking. DDP, by mixing linear and second-order information introduces additional $\alpha$-sensitivity that can cause excessive cooling of the line-search parameter. Thus, reducing the constant $c$ of the contraction map in eq.~\eqref{eq:con_map}.  

Figures \ref{fig:cool_pendulum}, \ref{fig:cool_cartpole}, \ref{fig:cool_TP1}, and \ref{fig:cool_TP2} contrast the variation of cost and learning rate, $\alpha$, for iLQR and DDP over successive iterations. As seen in Fig.~\ref{fig:cool_pendulum},  for the pendulum swing up, the learning rate remains fixed at $1$ for iLQR, while for DDP it drops to smaller values after iteration 4. Even during the first four iterations, when $\alpha = 1$ for both methods, iLQR achieves lower costs than DDP. In addition, DDP exhibits a noticeable cost increase from iteration 9 to 10, demanding a need for regularization. Likewise, for cartpole swing up, in Fig.~\ref{fig:cool_cartpole}, the learning rate for DDP decreases to nearly zero after a few iterations, resulting in a very slow cost reduction in subsequent steps. By contrast, iLQR maintains a learning rate of about $0.1$ and reduces the cost more effectively.

\begin{figure}[!htbp]
    \centering
    \subfloat[Pendulum swing-up.\label{subfig:comp_pendulum}]{\includegraphics[width=.48\linewidth]{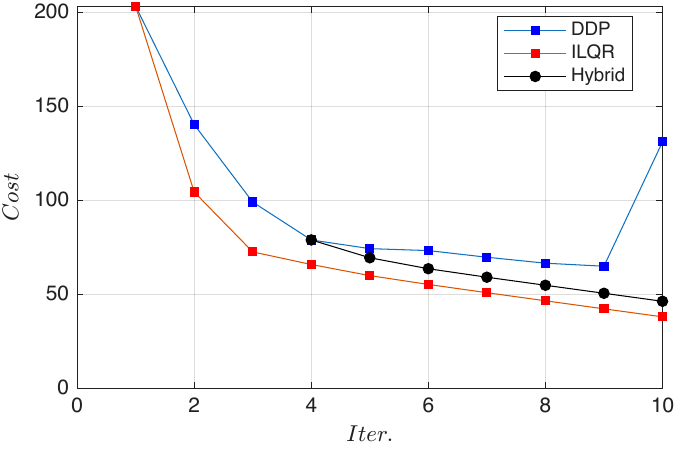}}
   \hfill
   \subfloat[Cart-pole swing-up.\label{subfig:comp_cartpole}] {\includegraphics[width=0.48\linewidth]{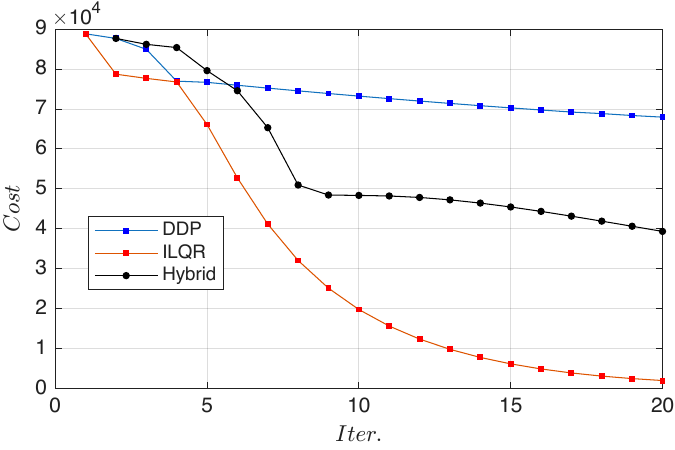}}%
   \hfill
   \subfloat[Test Problem-1.\label{subfig:comp_TP1}] {\includegraphics[width=0.48\linewidth]{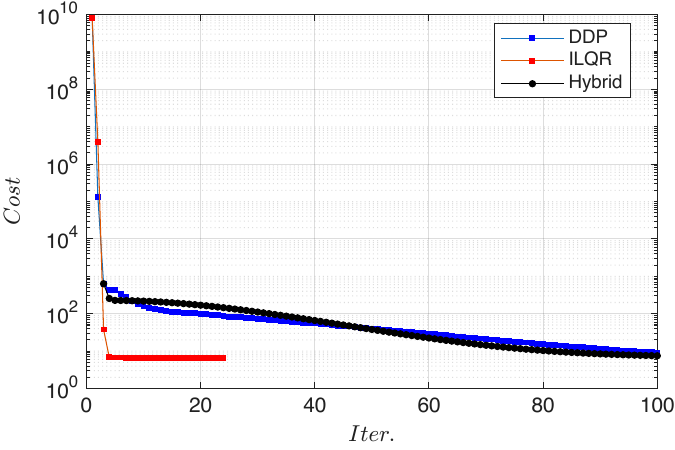}}
   \hfill
   \subfloat[Test Problem-2.\label{subfig:comp_TP2}] {\includegraphics[width=0.48\linewidth]{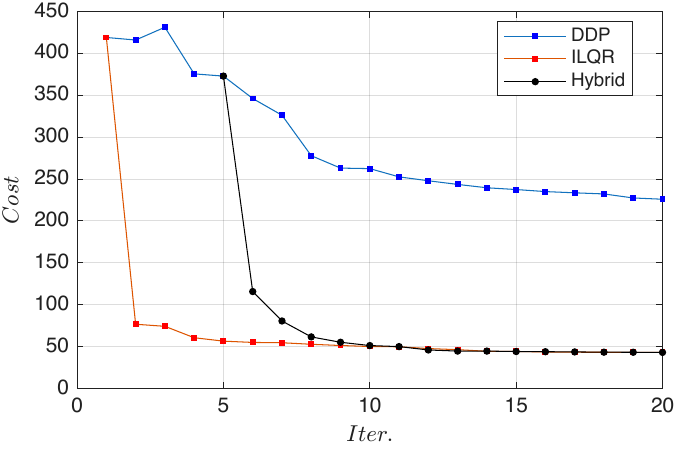}}
    \caption{A comparison of iLQR, DDP and hybrid solution. The hybrid solution depicts iLQR applied to the point where DDP solution reduces $\alpha$ (depicted in Figures \ref{fig:cool_pendulum}, \ref{fig:cool_cartpole}, \ref{fig:cool_TP1} and \ref{fig:cool_TP2}).}
    \label{fig:correction_lr}
\end{figure}

We applied iLQR at the point where DDP significantly reduces the line-search parameter. The resulting performance is shown in Fig.~\ref{fig:correction_lr}, where it is labeled as ``Hybrid.'' The plots clearly indicate that switching to iLQR at the stage where DDP cools down improves performance. Unlike DDP, iLQR maintains a stable learning rate and shows better convergence for the cart-pole, pendulum swing-up tasks and Test Problem-2. It did show similar convergence as compared to DDP for Test Problem-1. It should be noted that the  cost increased in both Table \ref{tab:cost_pred_DDP_pen} and Table \ref{tab:cost_pred_DDP} as we used an unregularized version of DDP to test this.

\begin{table}
    \centering
    \begin{tabular}{|c|c|c|c|}
    \hline
      Iter.   & $J$ & $\Delta J$ & $J_{\text{pred}} = J + \Delta J$\\
      \hline
        0 & $701.4661$ & $-377.7732$ & $323.6929$ \\
        \hline
        2 & $149.8840$ & $-51.4378$ & $98.4462$ \\
        \hline
        5 & $2.872408 \times 10^{5}$ & $-3.3685 \times 10^{5}$ & $-4.9609 \times 10^{4}$ \\
    \hline
    \end{tabular}
    \caption{Change in cost predicted by DDP over various iterations starting from random initial guesses for the pendulum swing-up task.}
    \label{tab:cost_pred_DDP_pen}
\end{table}

\begin{table}
    \centering
    \begin{tabular}{|c|c|c|c|}
    \hline
      Iter.   & $J$ & $\Delta J$ & $J_{\text{pred}} = J + \Delta J$\\
      \hline
        0 & $8.897408 \times 10^{5}$ & $-8.1926 \times 10^{5}$ & $7.0481 \times 10^{4}$ \\
        \hline
        3 & $6.011446 \times 10^{5}$ & $-9.3269e \times 10^{4}$ & $5.0788 \times 10^{5}$ \\
        \hline
        5 & $4.998068 \times 10^{5}$ & $-3.6467 \times 10^{7}$ & $-3.5968 \times 10^{7}$ \\
        \hline
    \end{tabular}
    \caption{Change in cost predicted by DDP over various iterations starting from random initial guesses for the cart-pole swing-up task.}
    \label{tab:cost_pred_DDP}
\end{table}

Another symptom of using DDP away from the optimal is an unrealistic cost decrease.
Table \ref{tab:cost_pred_DDP} shows cost reduction prediction for pendulum swing-up task. It indicates that at iteration 5, the cost predicted by DDP becomes negative ($-3.6\times 10^{7}$). Since the experiment used a purely quadratic cost with a lower bound of $0$ ($J_{\text{min}} = 0$), such a negative prediction is not feasible. This issue does not occur with iLQR. 

A similar problem is shown in Table \ref{tab:cost_pred_DDP_pen}, which shows the same problem for the pendulum swing-up task. DDP predicts a change in cost $\Delta J$ such that $J + \Delta J < J_{min}$. It should also be noted that the cost at iteration 5 is more than that of previous iterations, as we are considering unregularized DDP, and it tends to increase the cost because of $Q_{uu}$ being non-positive definite. 

\subsection{Faster convergence of DDP over iLQR near optima}
In certain scenarios, DDP demonstrates faster convergence than iLQR. In particular, near the optimal solution, a DDP step effectively shows a quadratic convergence without reducing the constant $c$ in eq.~\eqref{eq:con_map} to very low values. As a result, DDP often takes larger and more effective steps in proximity to the optimum, leading to quicker convergence in many cases.

\begin{figure}[!htbp]
    \centering
    \subfloat[Pendulum swing-up.\label{subfig:DDP_pen}]{
    \includegraphics[width=.45\linewidth]{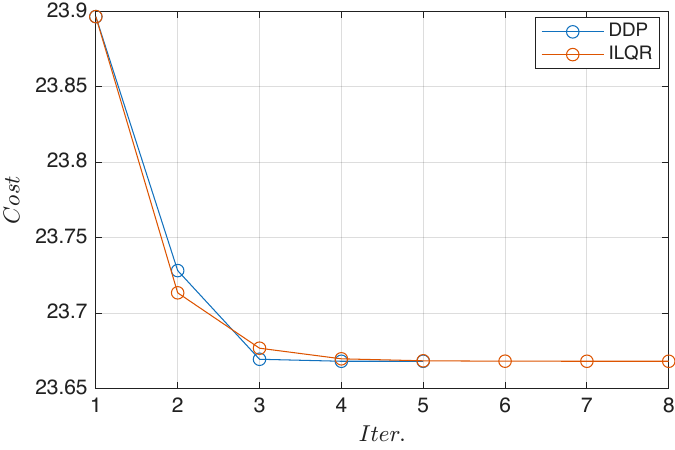}
    }
    \hfill
    \subfloat[Cart-pole swing-up\label{subfig:DDP_cart}]{
    \includegraphics[width=0.45\linewidth]{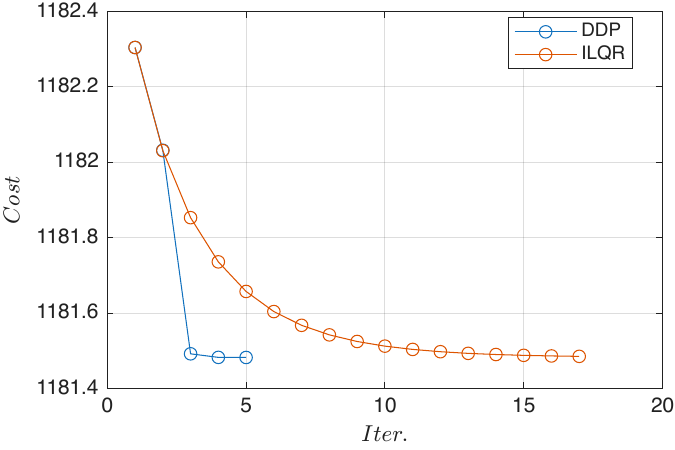}
    }
    \caption{Cost vs. number of iterations. With an initial control guess near the optimum, DDP reaches the optimum in fewer iterations.}
    \label{fig:DDP_fast}
\end{figure}

Figure \ref{fig:DDP_fast} illustrates that when initialized sufficiently close to the optimum, DDP tends to converge in fewer iterations by taking larger steps. In Fig.~\ref{subfig:DDP_pen}, for the pendulum swing-up task, DDP converges within 5 iterations, whereas iLQR requires 8. Similarly, in Fig.~\ref{subfig:DDP_cart}, for the cartpole swing-up, DDP converges in 5 iterations compared to 17 iterations for iLQR. Thus, in the neighborhood of an optimal solution, DDP typically achieves faster convergence than iLQR. DDP otperorming iLQR is also seen in Fig.~\ref{fig:cool_TP1} in iteration 2.
\begin{table}[h!]
\centering
\caption{Comparison of the number of iterations for iLQR, DDP, and Newton methods. The results are for Test Problem-3 with $n=100$, $m=50$, $T=100$, and $\mu = 1/200$.}
\begin{tabular}{|l|c|c|c|}
\hline
\textbf{} & \textbf{iLQR} & \textbf{DDP}\textsuperscript{b} & \textbf{Newton}\textsuperscript{b} \\ \hline
\textbf{Starting Point\textsuperscript{a}} & \textbf{No. of iter.} & \textbf{No. of iter.} & \textbf{No. of iter.} \\ \hline \hline
1 & 8 & 8  & 8 \\
2 & 11 & 9 & 10 \\
3 & 5 & 5 & 6 \\
4 & 9 & 9 & 9 \\
5 & 9 & 8 & 9 \\ \hline
\textbf{Average} & \textbf{8.4} & \textbf{7.8} & \textbf{8.4} \\ \hline
\multicolumn{4}{l}{\textsuperscript{a}\footnotesize{Initial control guesses are defined in Table~\ref{tab:starting_points_3}.}}\\
\multicolumn{4}{l}{\textsuperscript{b}\footnotesize{Data taken from \cite{LiaoShoemaker1992}}}\\
\end{tabular}

\label{tab:iterations}
\end{table}

\begin{table}[htbp]
    \centering
    \caption{Initial control trajectories for Test Problem 3 ($T=100, m=50, n=100$).}
    \label{tab:starting_points_3}
    \begin{tabular}{|c|l|}
        \hline
        \textbf{Starting Point} & \textbf{Initial Control}, $\vu_t$ \\
        \hline
        1 & $\vu_{t} = \boldsymbol{0}$\\
        \hline
        2 & $\vu_{t} = 0.01 \cdot \boldsymbol{1} $\\
        \hline
        3 & $\vu_t = -0.01 \cdot \boldsymbol{1} $ \\
        \hline
        4 &
        $\begin{cases}
          \vu_{t} = 0.01 \cdot \boldsymbol{1} & t \text{ is odd} \\
          \vu_{t} = -0.01 \cdot \boldsymbol{1}  & t \text{ is even}
        \end{cases}
        $ \\
        \hline
        5 &
        $\begin{cases}
          \vu_{t} = -0.01 \cdot \boldsymbol{1} & t \text{ is odd} \\
          \vu_{t} = 0.01 \cdot \boldsymbol{1}   & t \text{ is even}
        \end{cases}
        $ \\
        \hline
    \end{tabular}
\end{table}
Table~\ref{tab:iterations} also shows that DDP might slightly outperform iLQR for Test Problem-3 with slight nonlinearity ($\mu = 1/200$).

\section{Neighboring Extremal Paths}\label{sec:neighbouring_optima}
If we solve the DDP equations at the optimal solution $\vu^{*}_t$, we will get the control update to be $\vK_{t}\delta \vx_{t}$ as the term $\alpha_{t}$ will be zero at the optimal. This feedback obtained by solving DDP equations is the optimal first order feedback. In this section, we show that this is the same as solving for neighboring extremal paths\cite{BrysonHo69}. 

To begin with let us start with the neighboring extremal equations. At the optimal, the first order perturbations of the cost are zero i.e, $\delta  J = 0$. So, the optimization is over the second order terms. Formally,
\begin{align}
    &\min_{\delta \vu_t} \delta^{2} \mathcal{J} = \frac{1}{2}\delta \vx_{T}\tran C_{T,\vx\vx}\delta \vx_{T} + \sum_{t=0}^{T-1}\frac{1}{2}\begin{bmatrix}
        \delta \vx_{t}\\
        \delta \vu_{t}
    \end{bmatrix}\tran \begin{bmatrix}
        H_{\vx_t\vx_t} && H_{\vx_t\vu_t}\\
        H_{\vu_t\vx_t} && H_{\vu_t\vu_t} 
    \end{bmatrix}\begin{bmatrix}
        \delta \vx_{t}\\
        \delta \vu_{t}
    \end{bmatrix} \nonumber \\
    &\text{subject to}: \vf_{\vx_t}\delta \vx_t + \vf_{\vu_t}\delta \vu_t = \delta \vx_{t+1} \nonumber \\
    &\delta \vx_{0} = 0
\end{align}
where $H$ is the Hamiltonian of the system given by
\begin{equation*}
    H = c(\vx_t,\vu_t) + \blambda_{t+1}\tran(\vf(\vx_t,\vu_t))
\end{equation*}

The solution to the above LQR problem is:
\begin{align}
    &\vS_{t} = H_{\vx_t,\vx_t} + \vf_{\vx_t}\tran \vS_{t+1}\vf_{\vx_t} -  \nonumber \\
    &(H_{\vu_t\vx_t} + \vf_{\vu_t}\tran\vS_{t+1}\vf_{\vx_t})\tran (H_{\vu_t\vu_t} + \vf_{\vu_t}\tran \vS_{t+1}\vf_{\vu_t})^{-1} (H_{\vu_t\vx_t} + \vf_{\vu_t}\tran \vS_{t+1}\vf_{\vx_t})\label{eq:new_back_pass}\\
    &\mathbf{K}_{t} = (H_{\vu_t\vu_t} + \vf_{\vu_t}\tran \vS_{t+1}\vf_{\vu_t})^{-1}(\vf_{\vu_t}\tran \vS_{t+1}\vf_{\vx_t} + H_{\vu_t\vx_t})\\
    &\delta \vu_{t} = -\mathbf{K}_{t}\delta \vx_{t}
\end{align}
In the above expression, we have
\begin{align*}
    &H_{\vx_t\vx_t} = c_{\vx_t\vx_t} + \blambda_{t+1} \otimes \vf_{\vx_t\vx_t} \\
    &H_{\vu_t\vx_t} = H_{\vx_t\vu_t}\tran = c_{u_t,x_t} + \blambda_{t+1} \otimes \vf_{\vu_t\vx_t} \\
    &H_{\vu_t\vu_t} = c_{\vu_t\vu_t} + \blambda_{t+1} \otimes \vf_{\vu_t\vu_t}
\end{align*}
This gives the backward pass in eq.~\eqref{eq:new_back_pass} to be:
\begin{align*}
    &\vS_{t} = \underbrace{c_{\vx_t\vx_t} + \blambda_{t+1} \otimes \vf_{\vx_t\vx_t} + \vf_{\vx_t}\tran \vS_{t+1}\vf_{\vx_t}}_{Q_{\vx\vx}^{t}} 
    - \underbrace{(c_{\vx_t\vu_t} + \blambda_{t+1} \otimes \vf_{\vx_t\vu_t} + \vf_{\vx_t}\tran \vS_{t+1}\vf_{\vu_t})}_{Q_{\vx\vu}^{t}} \nonumber \\ &{\underbrace{(c_{\vu_t\vu_t} + \blambda_{t+1} \otimes \vf_{\vu_t\vu_t} + \vf_{\vu_t}\tran \vS_{t+1}\vf_{\vu_t})}_{Q_{\vu\vu}^{t}}}^{-1} \underbrace{(c_{\vu_t\vx_t} + \blambda_{t+1} \otimes \vf_{\vu_t\vx_t} + \vf_{\vu_t}\tran \vS_{t+1}\vf_{\vx_t})}_{Q_{\vu\vx}^{t}}. 
\end{align*}
This backward pass is the same as the backward pass given in the equation \eqref{eq:back_pass_DDP}. In addition to this, we have
\begin{align*}
    &\delta \vu_{t} = -\mathbf{K}_{t}\delta \vx_{t}\\
    &\mathbf{K}_{t} = {\underbrace{(c_{\vu_t\vu_t} + \blambda_{t+1} \otimes \vf_{\vu_t\vu_t} + \vf_{\vu_t}\tran \vS_{t+1}\vf_{\vu_t})}_{Q_{\vu\vu}^{t}}}^{-1}\underbrace{(c_{\vu_t\vx_t} + \blambda_{t+1} \otimes \vf_{\vu_t\vx_t} + \vf_{\vu_t}\tran\vS_{t+1}\vf_{\vx_t})}_{Q_{\vu\vx}^{t}}.
\end{align*}
This is the same expression as obtained for the feedback $\vK_{t}$ for DDP. This is also the ``optimal first-order feedback."

We show the effect of this feedback stabilization by using it to stabilize the system under uncertainty. One should note that iLQR feedback is the same as DDP except for the tensor products, $\blambda_{t+1} \otimes \vf_{\vu_t\vu_t}$ and $\blambda_{t+1} \otimes \vf_{\vu_t\vx_t}$, which are absent in iLQR.
\begin{figure}[!htbp]
    \centering
    \subfloat[Pendulum swing-up.\label{subfig:feedback_pen}]{\includegraphics[width=.48\linewidth]{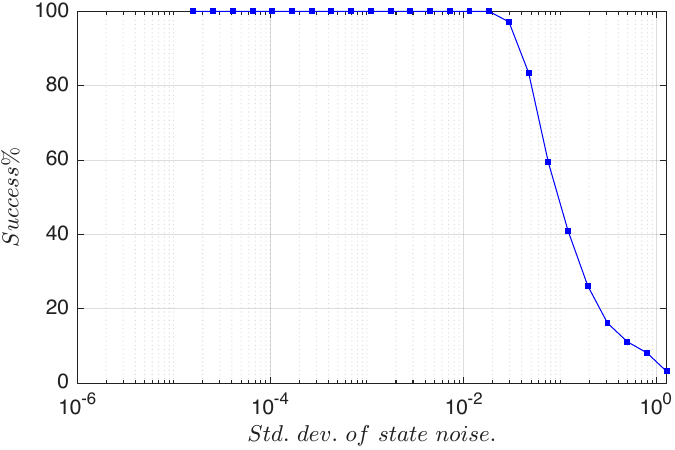}}
   \hfill
   \subfloat[Cart-pole swing-up.\label{subfig:feedback_cart}] {\includegraphics[width=0.48\linewidth]{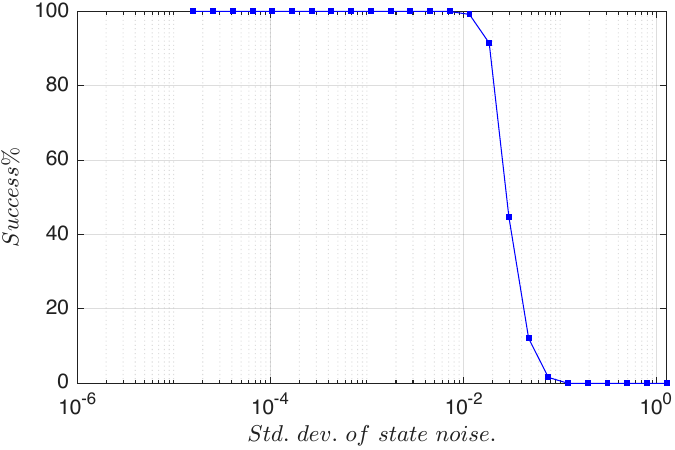}}%
    \caption{A plot showing stabilization under state noise with first-order optimal feedback for the cart-pole and pendulum swing up tasks.}
    \label{fig:feedback_comp}
\end{figure}

To show the efficacy of the feedback term, we introduce state noise in the state dynamics $\vx_{t+1} = \vf(\vx_{t},\vu_{t}) + \mathbf{w}_{t}$, where $\mathbf{w}_{t}$ was sampled from a gaussian distribution with zero mean and varying standard deviation. Fig. \ref{fig:feedback_comp} shows the feedback stabilization with varying noise. Success is defined when the feedback was able to take the system to a terminal step where $L_{2}$ the final state norm error was less than $0.1$. We observe how many trajectories satisfy this criteria over a total of 1000 simulations for each system. The x-axis in the figure indicates the standard deviation of the noise introduced. The first-order optimal feedback was able to stabilize the system perfectly upto a noise with standard deviation of 0.01 for the pendulum as well as the cart-pole swing up and then the performance started to decay.
\begin{figure}[!htbp]
    \centering
    \subfloat[$x$- the position of cart.\label{subfig:cart_pos_feed}]{\includegraphics[width=.48\linewidth]{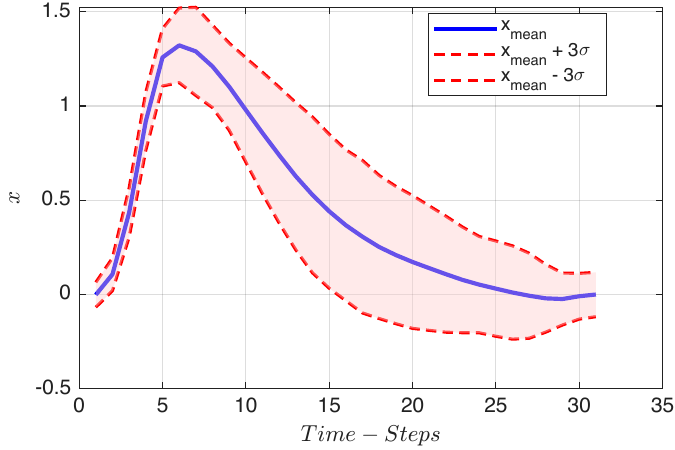}}
   \hfill
   \subfloat[$\dot{x}$- velocity of the cart.\label{subfig:cart_vel_feed}] {\includegraphics[width=0.48\linewidth]{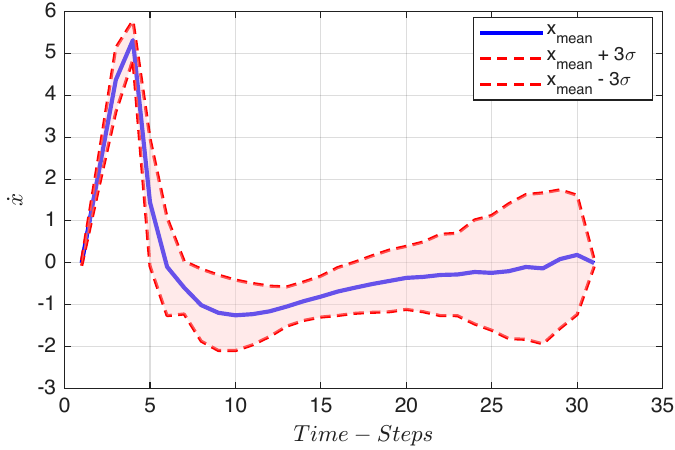}}%
   \hfill
   \subfloat[$\theta$ - angle of the pole.\label{subfig:pole_angle_feed}] {\includegraphics[width=0.48\linewidth]{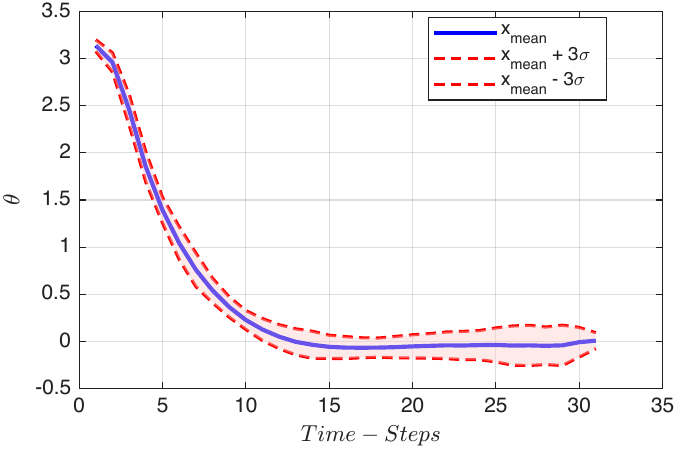}}
   \hfill
   \subfloat[$\dot{\theta}$ - angular velocity of the pole.\label{subfig:pole_vel_feed}] {\includegraphics[width=0.48\linewidth]{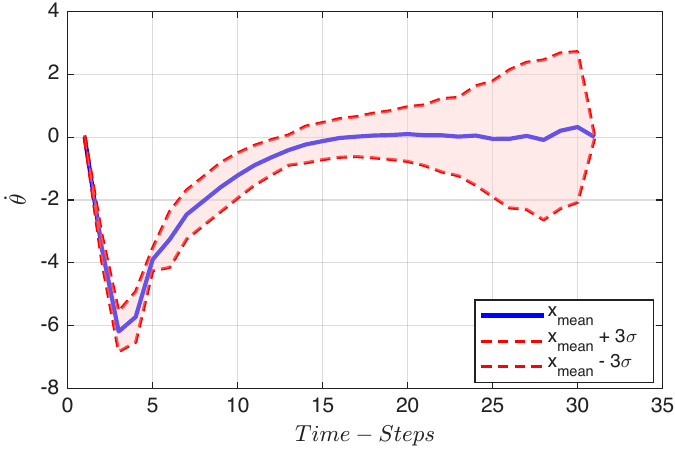}}
    \caption{Feedback stabilization of the cart-pole swing up with a gaussian noise (mean- $0$, standard deviation- $0.02$). The plot shows the mean and $3\sigma$ bound over $100$ trajectories.}
    \label{fig:cartpole_feedback_states}
\end{figure}
\begin{figure}[!htbp]
    \centering
    \subfloat[$\theta$ -angle of the pendulum.\label{subfig:theta_pen}]{\includegraphics[width=.48\linewidth]{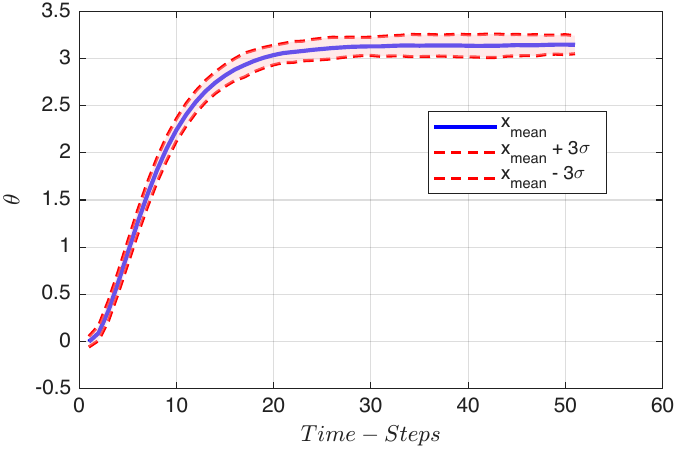}}
   \hfill
   \subfloat[$\dot{\theta}$ - angular velocity of the pendulum.\label{subfig:theta_dot_pen}] {\includegraphics[width=0.48\linewidth]{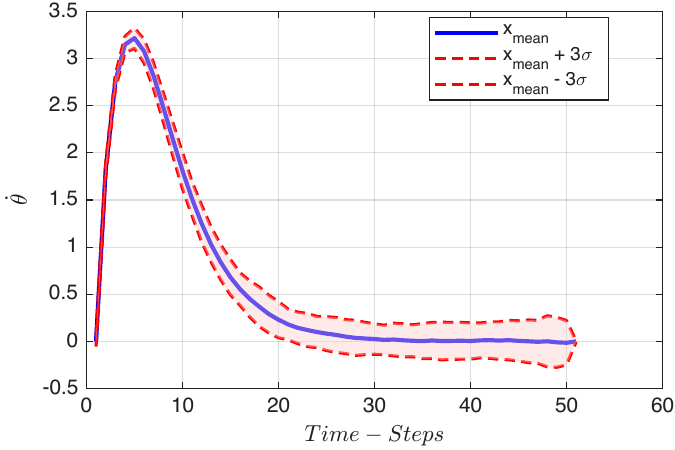}}%
   
    \caption{Feedback stabilization of the pendulum swing up with a gaussian noise (mean- $0$, standard deviation- $0.02$). The plot shows the mean and $3\sigma$ bound over $100$ trajectories.}
    \label{fig:pendulum_feedback_states}
\end{figure}

Figures \ref{fig:cartpole_feedback_states} and \ref{fig:pendulum_feedback_states} show state stabilization using the first-order optimal feedback. These experiments are run with a noise with mean $0$ and standard deviation $0.02$. 100 simulations were run to show the efficacy of the feedback stabilization in these systems.



\section{Conclusion}\label{sec:conclusion}

In this paper, we have studied the solution to discrete-time optimal control problems by analyzing Newton's method, Differential Dynamic Programming (DDP), and the iterative Linear Quadratic Regulator (iLQR) through a unified constrained Sequential Quadratic Programming (SQP) framework. The critical insight is that each of these methods solves a Quadratic Program (QP)—and consequently, an LQR problem—at every iteration.

Our analysis reveals that although DDP can be construed as an approximation to the Newton step near an optimum, both it and Newton's method can lead to inconsistent cost predictions and unreliable updates, especially when initialized far from the optimal solution. In contrast, iLQR provides consistent descent directions and admits theoretical convergence guarantees, making it far more robust in practice.

Experiments on multiple tasks confirm that iLQR delivers stable performance, whereas DDP may produce unreliable solutions unless heavily regularized. This regularization, however, removes any guarantee of convergence when far from an optimum. Indeed, iLQR demonstrated superior convergence over a regularized DDP implementation in one specific case ("Test Problem - 4"), while in other problems, the assumptions for iLQR to guarantee descent were satisfied, and it required no regularization.

These results establish iLQR as a principled SQP approach to optimal control rather than a mere approximation of DDP that neglects second-order terms. Our theoretical and empirical findings highlight the importance of consistency over second-order information; in fact, we show that omitting this second-order information is critical to performance, demonstrating an instance where less is more.

Additionally, we explored the neighboring optimal solution, showing it to be equivalent to expanding the perturbed Bellman equation around the optimum. This equivalence yields the DDP feedback solution near the optimum, which can be used to suppress noise and stabilize the system in the presence of uncertainty.


\bibliographystyle{plain}        
\bibliography{references}           



\end{document}